\documentclass[11pt]{amsart}
\usepackage[latin1]{inputenc}
 \usepackage{cite}
 \usepackage{relsize}
 \usepackage{bibentry}
 \usepackage{graphicx}
 \usepackage[all]{xypic}
 \usepackage{etoolbox}%
\usepackage{amsmath, amsfonts, amssymb, amsthm, mathtools, mathrsfs,bm}
\usepackage{enumerate}
\usepackage[numbers]{natbib}
\usepackage{setspace}

\newcommand{\nsh}{{n\,\sharp}}
\def\Hom{\mathop{\rm Hom}\nolimits}

\def\uK{\underline{K}}

\usepackage[pagebackref,colorlinks,linkcolor=red,citecolor=blue,urlcolor=blue,hypertexnames=true]{hyperref}
\usepackage[pagebackref,hypertexnames=true]{hyperref}

\newtheorem{theorem}{Theorem}[section]
\newtheorem{lemma}[theorem]{Lemma}
\newtheorem{proposition}[theorem]{Proposition}
\newtheorem{corollary}[theorem]{Corollary}

\theoremstyle{definition}
\newtheorem{definition}[theorem]{Definition}
\newtheorem{example}[theorem]{Example}
\newtheorem{remark}[theorem]{Remark}

\newcommand{\EG}{\underline{E}G}
\newcommand{\CR}{\mathcal{R}}
\newcommand{\QQ}{\mathcal{Q}}

\newcommand{\N}{\mathbb{N}}
\newcommand{\Z}{\mathbb{Z}}
\newcommand{\Q}{\mathbb{Q}}
\newcommand{\R}{\mathbb{R}}
\newcommand{\C}{\mathbb{C}}
\newcommand{\CCC}{\mathbb{C}}

\newcommand{\id}[1]{{\rm id}_{#1}}

\begin{document}
\title[Obstructions to matricial stability and almost flat K-theory]{Obstructions to matricial stability of discrete groups and almost flat K-theory}
\author{Marius Dadarlat}\address{MD: Department of Mathematics, Purdue University, West Lafayette, IN 47907, USA}\email{mdd@purdue.edu}	
\begin{abstract}
A discrete countable group $G$ is matricially stable if the finite dimensional approximate unitary representations of $G$ are perturbable to genuine representations in the point-norm topology.  For large classes of groups $G$, we show that matricial stability implies the vanishing of the rational cohomology  of $G$ in all nonzero even dimensions. We revisit a method of constructing almost flat K-theory classes of $BG$ which involves the dual assembly map and quasidiagonality properties of $G$. The existence of almost flat K-theory classes of $BG$ which are not flat represents an obstruction to matricial stability of $G$ due to continuity properties of the approximate monodromy correspondence.
\end{abstract}

\thanks{M.D. was partially supported by NSF grant \#DMS--1700086}
\maketitle
\section{Introduction}
The realization that there are K-theory obstructions  to perturbing approximate finite dimensional representations of $C^*$-algebras to genuine representations has emerged through work of Kazhdan \cite{Kaz-paper}, Voiculescu \cite{Voi:unitaries},
Connes, Gromov and Moscovici \cite{CGM:flat} and Connes and Higson \cite{Con-Hig:etheory}.

In the realm of groups, it is natural to ask when an approximate representation is close to a genuine representation. The reader is referred to the survey papers by  Arzhantseva \cite{Arzhantseva} and  Thom \cite{Thom:ICM} for a background discussion and more context for this question.
All norms considered in this paper refer to the uniform operator norm, $\|T\|=\sup_{\|\xi\|\leq 1} \|T\xi\|$. $M_n$ will denote the $n\times n$ complex matrices.
We consider only countable discrete groups $G$. By a group representation we will always mean a unitary representation.
A sequence of unital maps \mbox{$\{\varphi_n:G \to U(k_n)\}_{n\in \N}$} is called an asymptotic homomorphism if
\begin{equation}\label{ah}
  \lim_{n\to \infty} \|\varphi_n(st)-\varphi_n(s)\varphi_n(t)\|=0, \,\text{for all}\, s, t \in G.
\end{equation}
 $G$ is an MF-\emph{group} \cite{CDE2013} or  a $(U(k_n), \|\cdot\|)_{n}$ approximated group  (in the terminology of  \cite{DECHIFFRE}) if there exists an asymptotic homomorphism that separates the elements of $G$ in the sense that
  \begin{equation}\label{sep}
   \limsup_{n\to \infty} \|\varphi_n(s)-1||>0, \,\text{for all}\, s\in G\setminus\{e\}.
\end{equation}

  It is an open problem to find examples of discrete countable groups which are not MF. Amenable groups, or more generally, groups that are locally embeddable in amenable groups (LEA) are MF by the breakthrough paper of Tikuisis, White and Winter \cite{TWW}.

A group $G$ is called $(U(k_n), \|\cdot\|)_{n}$-stable \cite{DECHIFFRE} or \emph{matricially stable}  \cite{ESS}  if for any  asymptotic homomorphism $\{\varphi_n:G \to U(k_n)\}_{n\in \N}$ (not necessarily a separating one) there is a sequence of homomorphisms \mbox{$\{\pi_n:G \to U(k_n)\}_{n\in \N}$} such that
$\lim_{n\to \infty} \|\varphi_n(s)-\pi_n(s)\|=0$ for all $s\in G.$
A systematic study of matricial stability was undertaken by Eilers, Shulman  and  S{\o}rensen in \cite{ESS}.

  In his seminal paper \cite{Voi:unitaries}, Voiculescu  showed that $\Z^2$ is not matricially stable by using a Fredholm index argument.  However, he suggested that this phenomenon is due to the nonvanishing of the 2-cohomology of $\mathbb{T}^2$. Since $B\Z^2=\mathbb{T}^2$ this amounts to the nonvanishing of $H^2(\Z^2,\Z)$. The role of K-theory in Voiculescu's example was highlighted by Exel and Loring \cite{Exel-Loring:inv=}.

The aim of this paper is to point out that the nonvanishing of even rational cohomology in positive dimensions is a first obstruction to matricial stability for large classes of discrete groups.

\begin{theorem}\label{thm:1}
Let $G$ be a countable discrete MF-group  that admits a $\gamma$-element (e.g. G is uniformly embeddable in a Hilbert space).
  If $H^{2k}(G,\Q)\neq 0$ for some $k\geq 1$,
  then $G$  is not matricially stable.
\end{theorem}
For the proof, we use quasidiagonality and KK-theory methods  \cite{AA}, \cite{Kubota2} in conjunction with the (rational) surjectivity of the dual assembly map \cite{Kas:inv}, \cite{Tu:gamma} to show that the rational K-theory of $BG$ is (locally) almost flat.
  Once that is achieved, Theorem~\ref{thm:1} is essentially a cohomological expression of the fact that a group $G$ cannot be matricially stable if  $BG$ admits almost flat K-theory classes which are not flat. Indeed, the approximate monodromy representations associated to these nontrivial almost flat classes, \cite{CGM:flat}, are not perturbable to genuine representations of $G$, due to continuity properties of the approximate monodromy correspondence \cite[Thm.3.3]{DadCarrion:almost_flat}.

We also consider completely positive  (cp) versions of the properties MF and matricial stability and show that the same cohomological obstructions  persists, by refining the argument outlined above.
A unital map $\varphi:G \to M_n$ is completely positive definite if for any finite set of elements $s_1,..,,s_r$ of $G$,  the matrix $(\varphi(s_i^{-1}s_j))$ is positive. By a theorem of Naimark, this is the case if and only if $\varphi$ extends to a unital completely positive (ucp) map $C^*(G)\to M_n$ on the $C^*$-algebra of $G$.

A group $G$ is \emph{weakly quasidiagonal} if there is a sequence $\{\varphi_n:G \to M_{k_n}\}_{n\in \N}$  of unital completely positive definite maps   which satisfies the conditions \eqref{ah} and  \eqref{sep}, see also Definition~\ref{def:wqd}.
A group $G$ is \emph{weakly matricially stable} if for any sequence $\{\varphi_n:G \to M_{k_n}\}_{n\in \N}$  of unital completely positive definite maps which satisfies the condition \eqref{ah}, there are two sequences of homomorphisms \mbox{$\{\pi^{(i)}_n:G \to U(k^{(i)}_n)\}_{n\in \N}$}, $i=0,1$ such that
$\lim_{n\to \infty} \|\varphi_n(s)\oplus\pi^{(0)}_n(s)-\pi^{(1)}_n(s)\|=0$ for all $s\in G.$

 It is  clear from definitions that MF $\Rightarrow$ weak quasidiagonality and matricial stability $\Rightarrow$ weak matricial stability.
A group $G$ is \emph{quasidiagonal}  if it is isomorphic to a subgroup of the unitary group of a quasidiagonal $C^*$-algebra, see Definition~\ref{def:qd}.
Quasidiagonal groups are  weakly quasidiagonal. A group is maximally almost periodic (MAP) if its finite dimensional  unitary representations separate the points. MAP groups and residually amenable groups are quasidiagonal by \cite{TWW}.
 It is immediate from definition that a matricially stable MF-group must be MAP. In view of this discussion, it follows that  Theorem~\ref{thm:1} is a consequence
of a more general result which shows that nonvanishing of the rational even cohomology also obstructs weak matricial stability:
\begin{theorem}\label{thm:2}
Let $G$ be a countable discrete weakly quasidiagonal group  that admits a $\gamma$-element.
  If $G$ is weakly matricially stable, then $G$ is MAP and $H^{2k}(G,\Q)=0$ for all $k\geq 1$.
\end{theorem}
In Example~\ref{rem:contraex}, we indicate a 3-dimensional Bieberbach group which is weakly matricially stable but not matricially stable.
 A linear group is a subgroup of $GL_n(F)$ where $F$ is a field.
\begin{corollary}\label{cor:linear}
  If $G$ is a countable linear group such that $H^{2k}(G,\Q)\neq 0$ for some $k\geq 1$,
  then $G$ is not weakly matricially stable.
\end{corollary}
More examples of groups which are not matricially stable are given in Section~\ref{Sect:ex}.
It will follow from the proof of Theorem~\ref{thm:2} that if the classifying space $BG$ is a finite simplicial complex and $G$ is weakly matricially stable then $K^0(BG)$ must be generated by flat bundles, Remark~\ref{rem:flat}.
If  $\gamma=1$ (as is the case for  groups with Haagerup's property \cite{HigKas:BC}),  the Baum-Connes map is an isomorphism and so Theorem~\ref{thm:2} implies that $K_0(C^*(G))\otimes \Q\cong\Q$ whenever $G$ is torsion free.
In contrast, while any finite group $G$ is matricially stable,  $K^0(C^*(G))\cong K_0(C^*(G))\cong \Z^n,$ where $n$ is the number of equivalence classes of  irreducible representations of $G$. Kasparov introduced the ring $KK_G(\C,\C)\cong K^0(C^*(G))$ as a generalization of the representation ring of finite (compact) groups. Let us denote by $R(G)_{fin}$ the subring of $K^0(C^*(G))$ generated by finite dimensional unitary representations.
For the sake of simplicity, we will illustrate how one can obtain somewhat stronger obstructions to matrix stability  in the case of  groups with Haagerup's property.
\begin{theorem}\label{thm:3}
Let $G$ be a countable discrete weakly quasidiagonal group with  Haagerup's property. Suppose that $G$ is weakly matricially stable.
Then $G$ is MAP and:

(i) If $K_*(C^*(G))$ is finitely generated, then $K^0(C^*(G))=R(G)_{fin}$

(ii) If $G$ is torsion free, then  $K_0(C^*(G)) \otimes \Q \cong \Q$.

(iii) If  $BG$ is a finite simplicial complex, then  $K^0(BG)$ is generated by flat bundles.
\end{theorem}
 If  a group $G$ with Haagerup's property admits a $G$-compact model of $\underline{E}G$ then $K_*(C^*(G))$ is finitely generated, see \cite{Proietti}. If we assume moreover that $BG$ is a finite simplicial complex, then $K^0(C^*(G))=R(G)_{fin}$ if and only if $K^0(BG)$ is generated by flat bundles, see Remark~\ref{rem:flat}.
\begin{theorem}\label{thm:amen}
  Let $G$ be an amenable countable discrete such that $K_*(C^*(G))$ is finitely generated.
  Then $G$ is weakly matricially stable if and only if $G$ is MAP and $K^0(C^*(G))=R(G)_{fin}$.
\end{theorem}

In a recent paper, Ioana, Spaas and Wiersma \cite{Ioana-S-W:coho} use nonvanishing of 2-cohomology groups in conjunction with the relative Property (T) to exhibit obstructions to cp-lifting for full group $C^*$-algebras.
In particular it follows from their arguments that if $G=\Z^2\rtimes SL_2(\Z)$ (or if $G$ is any other group that satisfies the assumptions of their Theorem A) then there is a unital $*$-homomorphism $C^*(G) \to  \prod_n M_n/\bigoplus_n M_n$ which does not even have a ucp lifting  and hence $G$ is far from being matricially stable.
However, it should be noted that the absence of a ucp lifting is only one aspect of nonstability. Indeed, Corollary~\ref{cor:linear} shows that  there are $*$-homomorphisms $C^*(\Z^2\rtimes SL_2(\Z)) \to  \prod_n M_n/\bigoplus_n M_n$ which have ucp liftings but do not  lift to $*$-homomorphisms.

In Section~\ref{3} we discuss some basic aspects of MF,  weakly quasidiagonal and quasidiagonal groups as these are the classes of groups that concern our main results.
The proof of Theorem~\ref{thm:2} is given in Section~\ref{4} and
the proofs of Theorems~\ref{thm:3} and \ref{thm:amen} are given in Section~\ref{sec:msa}.
We rely on ideas and techniques developed in connection with the strong Novikov conjecture and the Baum-Connes conjecture \cite{BCH}, \cite{Kas:inv}, \cite{Kasparov-Skandalis-kk}, \cite{Kasparov-Skandalis-cr}, \cite{CGM:flat}, \cite{Yu:BC}, \cite{HigKas:BC}, \cite{Chabert-Echterhoff-documenta}, \cite{Ska-Tu-Yu:BC}, \cite{Tu:gamma}.
The statement of Theorem~\ref{thm:1} for groups $G$ with compact classifying space and $C^*(G)$  quasidiagonal  can be derived from our earlier paper \cite{AA}. In view of \cite[Prop.3.2]{AA}, the papers \cite{AA}, \cite{BB}, \cite{Carrion-Dadarlat}, \cite{Dad-Pennig-homotopy-symm}, \cite{Dad-Pennig-connective}  implicitly exhibit  classes of groups which are not matricially stable.
Just like in \cite{AA}, we employ here the concept of quasidiagonality to produce local approximations of the dual assembly map, but additionally, we use a key idea of Kubota \cite{Kubota2} and consider quasidiagonal $C^*$-algebras which are intermediate between  the full and the reduced group $C^*$-algebras  to enlarge the class of groups for which these approximation methods are applicable.
In Section~\ref{5},
we revisit briefly the KK-theoretic  approach to almost flat K-theory that we introduced in \cite{AA}. In view of \cite{Kubota2}, this method extends now to the class of quasidiagonal groups. The problem of constructing almost flat classes on classifying spaces was considered first by Connes, Gromov and Moscovici \cite{CGM:flat} and Gromov in \cite{Gromov:reflections, Gromov:curvature}. For now, it seems that the use of quasidiagonal  KK-theory classes remains the most effective approach to this problem.

The class of groups that admit a $\gamma$-element is very large as shown by work of Kasparov and Skandalis \cite{Kasparov-Skandalis-kk}, \cite{Kasparov-Skandalis-Annals}.
Moreover, J.-L.Tu has shown that if $G$ is a discrete countable group that admits a uniform embedding in a Hilbert space, then $G$ admits a $\gamma$-element \cite{Tu:gamma}.
The amenable groups, or more generally, the groups with Haagerup's property are uniformly embeddable in a Hilbert  space \cite{Valette-book}  and so are the linear groups as shown by Guentner, Higson and Weinberger \cite{GuenHW}.
The class of groups that admit a uniform embedding in a Hilbert space
 is closed  under  subgroups  and  products,  direct  limits,  free  products with amalgam, and extensions by exact groups \cite{Dad-Guen}.

 The author is grateful to the referee of this paper for useful comments.
\section{Examples}\label{Sect:ex}
In this section we exhibit  classes of groups which are not (weakly) matricially stable by virtue of Theorems~\ref{thm:1} and ~\ref{thm:2}.
In view of our earlier discussion, it is not surprising that the list of matricially stable groups cannot be too extensive. It was shown in \cite{ESS} that all finitely generated virtually free groups are matricially stable. In particular it follows that $SL_2(\Z)$ is matricially stable. Moreover, it was established in  \cite{ESS} that the only crystallographic groups that are matricially stable are the two line groups $\Z$ and $\Z/2*\Z/2$ and the 12 wallpaper groups that contain at least one reflection or one glide reflection.

\noindent\textbf{Proof of Corollary~\ref{cor:linear}}.
 Any  linear group $G$ is uniformly embeddable in a Hilbert space by \cite{GuenHW}.
Write $G$ as the increasing union of a sequence  $(G_n)_n$ of finitely generated subgroups.
Each $G_n$ is residually finite by a classic theorem of Malcev \cite[6.4.13]{Brown-Ozawa}. It follows then that $G$ is weakly quasidiagonal by Proposition~\ref{prop:cp-MF}.
 We conclude the proof  by applying Theorem~\ref{thm:2}.\qed
\begin{corollary}\label{cor:hyperbolic}
Let $G$ be a hyperbolic residually finite group with $H^{2k}(G;\Q)\neq 0$ for some $k\geq 1$.
Then $G$ is not weakly matricially stable.
\end{corollary}
\begin{proof} Hyperbolic groups are exact \cite{Brown-Ozawa} and hence they admit a $\gamma$-element by \cite{Tu:gamma}. Thus the statement follows from Theorem~\ref{thm:2}.
It is not known whether there exists a hyperbolic group which is not residually finite.
\end{proof}
The following generalizes a result obtained in \cite{ESS} for two-step nilpotent groups.
\begin{corollary}\label{cor:nilpotent} $\Z$ is the only nontrivial matricially stable torsion free finitely generated nilpotent group.
\end{corollary}
\begin{proof}
  Let $G$ be a nontrivial torsion free  finitely generated nilpotent group. Malcev has shown that $G$ is linear and that one can associate to $G$ a finite dimensional  rational Lie algebra of dimension equal to the Hirsch number of $G$. By \cite{Pickel}, $H^*(G;\Q)\cong H^*(L;\Q)$. But $H^2(L,\Q)\neq 0$ if $L$  has dimension $>1$ by a result of Ado as explained in \cite[p.86]{Chevalley-Eilenberg}. Thus if $G$ is matricially stable, then $L=\Q$ and hence $G=\Z$.
\end{proof}
 Blackadar \cite{Bla:shape} proved that matrices over semiprojective unital $C^*$-algebras are semiprojective. Matricial stability of $G$ can be interpreted as matricial weak semiprojectivity of $C^*(G)$, \cite{ESS}.
 Using the same arguments as in \cite{Bla:shape}, one shows that
if $G$ is matricially stable  and $H$ is  finite, then $G\times H$ is matricially stable since $C^*(G\times H)\cong C^*(G)\otimes C^*(H)$ and $C^*(H)$ is a finite dimensional $C^*$-algebra.
The next corollary notes that
up to rational cohomology these are the only possible examples of matricially stable direct products.
\begin{corollary}\label{cor:product}
Let $G_1,G_2$ be countable discrete MF-groups  that admit $\gamma$-elements. If $G_1\times G_2$ is matricially stable,
then one of the two groups must have the rational cohomology of a trivial group.
\end{corollary}
\begin{proof} It is immediate that both groups must be matricially stable and hence $H^{2k}(G_i;\Q)=0$, $i=1,2$, for $k\geq 1$.
Seeking a contradiction, suppose that there are $k_i\geq 0$ such that $H^{2k_i+1}(G_i;\Q)\neq 0$ for $i=1,2$.
The K\"{u}nneth formula implies that $H^{2k_1+2k_2+2}(G_1\times G_2;\Q)\neq 0$ which contradicts Theorem~\ref{thm:1}.
\end{proof}

\noindent\textbf{One-relator groups}.
A one-relator group is a group with a presentation of the form $\langle S; r \rangle$, where $r$ is a single element in the free group $F(S)$ on the countable generating set $S$. An important example is the surface group
 \[\Gamma_g=\pi_1(S_g)=\langle s_1,t_1,...,s_g,t_g\,;\, \prod_{i=1}^g [s_i,t_i]\,\rangle,\]
where  $S_g$ is  a connected closed orientable surface of genus $g\geq 1$.
Kazhdan \cite{Kaz-paper} showed that $\Gamma_g$ is not matricially stable. This was reproven (implicitly) in \cite{BB} and in \cite{ESS}. We revisit $\Gamma_g$ in Corollary~\ref{cor:surface}.

 Another important class of examples consists of the Baumslag-Solitar groups
\[BS(m,n)=\langle s,t\, ; \,  st^m s^{-1}t^{-n}\rangle, \quad m, n \in \Z\setminus\{0\}.\]

In general, one can always write $G=\langle S; s^n \rangle$ where $s$ is not a proper power of an element of $F(S)$ and $n\geq 1$.
By a theorem of Karrass, Magnus and Solitar \cite[Thm.5.2]{Baumslag-survey}, $G$ is torsion free if and only if $n=1$.
By a theorem of Wise \cite{Wise}, all torsion one-relator groups are residually finite (and hence MF).
It is not known if all one-relator groups are MF or hyperlinear.

\begin{corollary}\label{cor:surface}
Let $G=\langle S; s^n \rangle$ be a one-relator group.
\begin{itemize}
  \item[(i)] If $G$ is MF (this is automatic if $n>1$) and  $s\in [F(S),F(S)]$ i.e. the image of $s$ is trivial in $G_{ab}$,  then $G$ is not matricially stable.
  \item[(ii)] The surface groups $\Gamma_g$ are not weakly matricially stable for $g\geq 1$.
  \item[(iii)] The Baumslag-Solitar groups $BS(m,n)$ are not weakly matricially stable if $|m|, |n| \geq 2$ and either $|m|\neq |n|$ or $m=n$.
  \end{itemize}
\end{corollary}
\begin{proof}
  (i) It is known that $H^2(G;\Q)=\Q$ if  $s\in [F(S),F(S)]$ and $H^2(G;\Q)=0$ if $s\notin [F(S),F(S)]$. This is explained in \cite[p.279]{Valette-one} based on \cite{Baumslag-survey}. Since  one-relator groups are exact by \cite{Tu:notes}, \cite{Guentner-one} we can invoke Theorem~\ref{thm:1}.

  (ii) $\Gamma_g$ are linear groups as a consequence of the uniformization theorem.
   A very short proof that surface groups are residually finite was given by Hempel ~\cite{Hempel}.
  The conclusion follows from (i).

  (iii)  It was shown by Kropholler \cite{Kropholler} that $BS(m,n)$ has the second derived group free and hence it is residually solvable as noted for example in \cite[Ex.1.1]{Clair}. Since  $BS(m,n)$ is residually solvable, it is residually amenable and hence it is a quasidiagonal group by \cite{TWW}.
   On the other hand, it was shown by Meskin \cite{Meskin} that $BS(m,n)$ is residually finite if and only if $|m|=1$ or $|n|=1$ or $|m|= |n|$.
    Thus $BS(m,n)$ is not residually finite if $|m|, |n| \geq 2$ and  $|m|\neq |n|$ and hence it is not MAP.  It follows that it cannot be weakly matricially stable. The case $m=n$ follows from (i) since $st^m s^{-1}t^{-n}$ is a commutator.i
\end{proof}

The groups $B(m,m)$ and  $B(2,3)$ were already shown not to be matricially stable in \cite{ESS}. $B(1,-1)\cong B(-1,1)$ is matricially stable by
  \cite{ESS}. The answer to matricial stability is unknown for the other Baumslag-Solitar groups.

  The mapping class group of $S_g$, denoted by $Mod(S_g)$, is defined as the group of the
isotopy classes of orientation-preserving diffeomorphisms \cite{MCG-book}.
\begin{corollary}\label{cor:mcg} The mapping class group $Mod(S_g)$ is not weakly matricially stable for $g\geq 3$.
\end{corollary}
\begin{proof}The assumptions of Theorem~\ref{thm:2} are verified due to the following important results.
$Mod(S_g)$ is an exact (boundary amenable) group by work of Kida \cite{Kida} and Hamenst\"{a}dt \cite{Ursula} and it is residually finite by \cite[Thm.6.1]{MCG-book}. In addition, $H^2(Mod(S_g);\Z)\cong \Z$ for $g\geq 4$ by a result of Harer \cite[Thm.5.8]{MCG-book} and $H^2(Mod(S_3);\Z)$ is isomorphic to either $\Z$ or $\Z\oplus \Z/2$ by \cite{g=3}.
\end{proof}
\begin{corollary}\label{cor:out} The groups $\mathrm{Aut}(\mathbb{F}_n)$ and $\mathrm{Out}(\mathbb{F}_n)$  are not weakly matricially stable for $n= 4,6,8$.
\end{corollary}
\begin{proof} $\mathrm{Aut}(\mathbb{F}_n)$ and $\mathrm{Out}(\mathbb{F}_n)$ are residually finite by \cite{Baumslag-auto} and \cite{Grossman}, respectively. The group $\mathrm{Out}(\mathbb{F}_n)$ is exact by \cite{Bestvina}. Since the center of $\mathbb{F}_n$ is trivial $\mathrm{Inn}(\mathbb{F}_n)\cong \mathbb{F}_n$ and hence $\mathrm{Aut}(\mathbb{F}_n)$ is exact being extension of an exact group by another exact group \cite{KirWas:permanence}. The cohomology of these groups is very hard to compute. Nevertheless it is known that  $H^{4k}(\mathrm{Aut}(\mathbb{F}_{2k+2});\Q))\neq 0 \neq H^{4k}(\mathrm{Out}(\mathbb{F}_{2k+2});\Q))$ for $k=1,2,3$. \cite{Vogtmann}.
\end{proof}
\begin{example}\label{rem:contraex}
 The vanishing of $H^{2k}(G;\Q)$, $k\geq 1$, is  only a first obstruction to matricial stability. Let $G$ be the 3-dimensional Bieberbach group
  \[G=\langle x,w,y\,:\, w^{-1}zw=x^{-1}zx=z^{-1}, \, w^{-1}xw =x^{-1}z\rangle.\]
  $G$ is not matricially stable, since as shown in \cite{ESS}, no crystallographic group of dimension $\geq 3$ is matricially stable.
   One knows that $H_1(G;\Z)=G_{ab}=\Z \times \Z/4$ by \cite{Conway-Rossetti} and $H^3(G,\Z)\cong \Z/2$ as the corresponding flat 3-manifold is not orientable \cite[p.120]{Wolf-spaces}. Since the Euler characteristic of a nontrivial torsion-free polycyclic-by-finite group $G$ is zero by \cite{Dekimpe}, one concludes that the rational cohomology of $G$
is given by: \(H^0(G;\Q)= H^1(G;\Q)=\Q,\,   H^2(G;\Q)= H^3(G;\Q)= 0.\)

We will argue that $G$ is weakly matricially stable. By Theorem~\ref{thm:amen} and Remark~\ref{rem:flat}, it suffices to show that $K^0(BG)$ is generated by flat bundles. Since $BG$ is 3-dimensional, $K^0(BG)$ is generated by line bundles and so $K^0(BG)\cong \Z \oplus H^2(BG,\Z).$
  By the universal coefficient theorem, $H^2(BG,\Z)\cong \mathrm{Ext}(H_1(BG,\Z),\Z)\cong \Z/4$ is a torsion group and hence its elements are represented by flat line bundles by \cite[2.6]{Putman}.
\end{example}
\section{Quasidiagonal Groups}\label{3}
In this section we discuss some very basic aspects of the classes of groups (MF, quasidiagonal and weakly quasidiagonal) that appear in the statements of our main results: \ref{thm:1}, \ref{thm:2} and \ref{thm:af}.
All the groups $G$ that we consider are countable and discrete.

Recall that a $C^*$-algebra is MF if it embeds in  $\prod_n M_{k_n}/\bigoplus_n M_{k_n}$, for some sequence $(k_n)$, \cite{BlackKir:1}.
\begin{definition}[\cite{CDE2013}]
  A group $G$ is MF if it is isomorphic to a subgroup of the unitary group of an MF-algebra.
\end{definition}
\begin{proposition}\label{Prop:MF}
  The following assertions are equivalent.
  \begin{itemize}
    \item[(i)] $G$ is an MF group
    \item[(ii)] $G$ embeds in $\mathbf{U}/\mathbf{N}$ where $\mathbf{U}=\prod_{n=1}^{\infty} U(n)$ and
  $\mathbf{N}=\{(u_n)_n \in \mathbf{U}\,:\, \|u_n-1_n\|\to 0\}$.
    \item[(iii)]  For each finite subset $F$ of $G$ and for any $\epsilon > 0$ there is a unital map
$\varphi:G \to U(n)$ such that $\|\varphi(st)-\varphi(s)\varphi(t)\|<\varepsilon$ for all $s,t\in F$ and  $\|\varphi(s)-1\|\geq \sqrt{2}$ for all $s\in F\setminus \{e\}$.
  \end{itemize}
\end{proposition}
\begin{proof} $(i) \Leftrightarrow (ii)$ is trivial whereas
  $(iii) \Leftrightarrow  (i)$ is proved in \cite{Korchagin}. See also Proposition~\ref{prop:char-cp-MF} for a related  property.
\end{proof}
\begin{remark}\label{rem:b}
(i)   Proposition~\ref{Prop:MF} implies that a group that is locally embeddable in  MF groups is  MF.  Since amenable groups are MF by \cite{TWW},
it follows  that  LEA groups are MF.

(ii) Thom \cite{Thom-hyperlinear} showed that a certain group $K$ constructed by de Cornulier is hyperlinear but not LEA. The group $K$ is the quotient of a property (T) subgroup of $SL_8(\Z[1/p])$ by a cyclic central subgroup.
If $N$ is a central subgroup of a residually finite group $G$, then $G/N$ is MF  by \cite[Thm.2.17]{CDE2013}.  Thus $K$ is also an example of an MF group which not LEA.
\end{remark}

\begin{definition}\label{def:qd} A countable discrete group $G$ is \emph{quasidiagonal} if it is isomorphic to a subgroup of the unitary group of a quasidiagonal $C^*$-algebra.
 \end{definition}

Equivalently, $G$ is quasidiagonal if it admits a quasidiagonal faithful unitary representation  on a separable Hilbert space. A representation $\pi:G \to U(H)$  is quasidiagonal if  there is an increasing sequence  $(p_n)_{n}$ of finite dimensional projections which converges strongly to $1_H$ and such that $\lim_{n\to \infty}\|[\pi(s),{p_n}]\|= 0$ for all $s \in G$. In other words, the $C^*$-algebra $C^*_\pi(G)=C^*(\pi(G))$ is quasidiagonal.

\begin{remark}\label{rem:c}

(i) If
$G\subset \prod_n U(B_n)$ with $B_n$ separable and quasidiagonal, then $G$ is quasidiagonal since the $C^*$-algebra $\prod_n B_n$
    is quasidiagonal. Thus the MAP groups are quasidiagonal.
  It also follows that the residually amenable groups are quasidiagonal by \cite{TWW}.

(ii) If $H$ is quasidiagonal and $G$ is amenable, then the  wreath product $H\wr G$ is quasidiagonal.
Indeed, by \cite[Thm.4.2]{Dad-Pennig-Schneider},
if $D$ is a unital separable quasidiagonal $C^*$-algebra and $G$ is a countable discrete amenable group, then the crossed product  $(\bigotimes_G D)\rtimes G$ is quasidiagonal (here we work with minimal tensor products and $G$ acts via noncommutative Bernoulli shifts).  Thus, if there is an embedding $\omega:H\to U(D)$, then the map $h\to 1_D \oplus \omega(h) \in U_2(D)$
induces an embedding  $\bigoplus_G H \to U(\bigotimes_G M_2(D))$. It follows that
 $H\wr G=\left(\bigoplus_G H\right)\rtimes G$ embeds in the unitary group of  $(\bigotimes_G M_2(D))\rtimes G$ and so $H\wr G$ is quasidiagonal.

(iii) The class of quasidiagonal groups is strictly larger than the class of residually amenable groups as noted in Example~\ref{qd-not-ra} below. We do not have examples of quasidiagonal groups which are not LEA.
\end{remark}
\begin{example}\label{qd-not-ra}
 Let $\mathrm{O}_n(\Q)\subset\mathrm{O}_n(\R)$ be the group of orthogonal matrices with rational entries.
 $\mathrm{O}_{n}(\Q)$ is MAP (since $\mathrm{O}_n(\R)$ is compact) and hence quasidiagonal. We will argue below that it is not residually amenable if $n\geq 5$.
  Let $\Omega_n$ denote the commutator subgroup of $\mathrm{O}_n(\Q)$. It was shown by
   Kneser \cite{Kneser} that the projective group $P\Omega_n$ is simple for $n\geq 5$. If $n$ is odd, the center of $\mathrm{O}_n(\Q)$ is trivial and in particular $\Omega_5=P\Omega_5$ is simple.
 Since $\Omega_5$ is simple and infinite, it is not virtually solvable and hence by the Tits alternative \cite{Tits} it contains a non-abelian free subgroup. 
It follows that any group that contains a subgroup isomorphic to $\Omega_5$ is not residually amenable. In particular this applies to $\mathrm{O}_n(\Q)$, $n\geq 5$.
\end{example}
 We denote by $\lambda_G$ the left regular representation of $G$ and by $\iota_G$  the trivial representation.
 \begin{proposition}\label{Bk}
 Let $\{\omega_n:G \to U(H_n)\}_{n\geq 1}$
 be  a  sequence of group
representations  that separates the points
of $G$.
Then $\lambda_G$ is weakly contained in
$\{\iota_G\}\cup\{ \omega_{i_1}\otimes  \cdots \otimes \omega_{i_n}\,\colon\, 1 \leq i_1\leq \cdots \leq i_n, \, n\geq 1\}.$
\end{proposition}
 \begin{proof}
Since $(\omega_n)_n$
separates the points of $G$, for each $s\in G\setminus\{e\}$ there is $i=i(s)\geq 1$ such that $\omega_i(s)\neq 1$. Using the spectral theorem we find a unit vector $\xi_i \in H_i$ such that the positive definite map associated to $\iota_G\oplus \omega_i$, $f_i(\cdot)=\frac{1}{2}(1+\langle\omega_i(\cdot)\xi_i,\xi_i\rangle)$ satisfies $|f_i(s)|<1$.
For $I=(i_1,\cdots, i_n)$ with $ 1 \leq i_1\leq \cdots \leq i_n$, the positive definite map $f_I=f_{i_1}\cdots f_{i_n}$ is associated to  the representation
 $(\iota_G\oplus \omega_{i_{1}})\otimes  \cdots \otimes (\iota_G\oplus \omega_{i_{n}})$ which is unitarily equivalent to a direct sum of representations from $\Omega:=\{\iota_G\}\cup\{ \omega_{i_1}\otimes  \cdots \otimes \omega_{i_n}\,\colon\, 1 \leq i_1\leq \cdots \leq i_n, \, n\geq 1\}.$ As $I$ increases in size, for each $s\neq e$, $f_I$ contains larger and larger powers of $f_{i(s)}$. Since $f_I(e)=1$, it follows that
 $\lim_{|I|\to \infty} f_I(s)=\delta_e(s)$ for all $s\in G$. Since the
positive definite map $\delta_e$ corresponds to a cyclic
vector of $\lambda_G$ and $f_I$ is associated to $\Omega$, it follows by \cite[18.1.4]{Dix:C*} that
$\lambda_G$ is weakly contained in $\Omega$.
 \end{proof}

\begin{proposition}\label{cor:qd-embed} Let $G$ be a countable discrete group. The following assertions are equivalent:
\begin{itemize}
  \item[(i)] $G$ is quasidiagonal.
  \item[(ii)] $\lambda_G$ is weakly contained in a quasidiagonal representation $\pi$ of $G$.
  \item[(iii)] The canonical  map $q_G:C^*(G)\to C_r^*(G)$ factors through  a unital  quasidiagonal C*-algebra.
\end{itemize}
\end{proposition}
\begin{proof}
   $(i) \Rightarrow (ii)$ By assumption, there is a faithful representation $\omega:G \to U(H)$ such that the $C^*$-algebra  $B=C_\omega^*(G)$
 is quasidiagonal.
   Proposition~\ref{Bk} shows that $\lambda_G$ is weakly contained in the set $\{\omega^{\otimes n}: \, n \geq 0\}$, $\omega^{\otimes 0}=\iota_G$, and hence in $\pi=\bigoplus_{n\geq 0} \omega^{\otimes n}.$
   Since $C_{\pi}^*(G)\subset \prod_{n\geq 0} B^{\otimes n}$ (minimal tensor products) and $B$ is quasidiagonal, the representation $\pi$ is quasidiagonal.

   $(ii) \Rightarrow (iii)$ If $\pi$ is as in (ii), then
    $q_G$ factors  through the quasidiagonal $C^*$-algebra  $C_{\pi}^*(G)$.

 $(iii) \Rightarrow (i)$
  By assumption, $q_G$ factors through  a unital  quasidiagonal C*-algebra $D$.
  Let $\pi$ be  the composition $G \to U(C^*(G))\to U(D)$.  Since $\pi$ is a lifting of $\lambda_G:G \to U(C^*_r(G))$, it follows that $\pi$ is injective. Moreover,
 $\|\pi(s)-\pi(t)\| \geq \|\lambda_G(s)-\lambda_G(t)\|\geq \sqrt{2}$ for $s\neq t$.
\end{proof}

Free products (without amalgamation) of quasidiagonal groups are quasidiagonal as a consequence of \cite{Boca:free-prod=qd}.
 It is an open problem whether the amagalmated product of two amenable groups over a finite group  is residually amenable or not \cite[Question 2]{Nikolov}. However, we show that it is quasidiagonal.
\begin{proposition}
 Let $G_1$, $G_2$ be residually amenable countable discrete groups with a common finite subgroup $H$. Then $G_1\star_{H} G_2$ is quasidiagonal.
\end{proposition}\label{prop:raqd}
\begin{proof} Let $G=G_1\star_{H} G_2$. As argued in \cite[p291]{Nikolov}, since $H$ is finite,  there exist sufficiently many homomorphisms $G \to G'_1\star_{H} G'_2$ with $H\subset G_i'$ and $G'_i$ amenable to separate the points of $G$. Thus it suffices to assume that  $G_1$ and $G_2$ are themselves amenable.
Let $\tau_i$ be the canonical trace on $C^*(G_i)\cong C^*_r(G_i)$ and let $\tau$ be the unique tracial state of universal UHF algebra $\QQ=\bigotimes_n M_n$. By Schafhauser's embedding theorem \cite[Thm.B]{Schaf:AF} there are unital $*$-monomorphisms $\varphi_i:C^*(G_i)\to \QQ$,   such that $\tau \circ \varphi_i=\tau_i$, $i=1,2$.
Since $\tau_i | _{C^*(H)}=\tau_H$ (the canonical trace on $C^*(H)$), it follows that $\tau | _{\varphi_1(C^*(H))}=\tau |_{\varphi_2(C^*(H))}$. By standard perturbation arguments one can find two increasing sequences $(A^{(i)}_n)_n$, $i=1,2$,  of  matrix subalgebras of $\QQ$ such that $\varphi_i(C^*(H))\subset A^{(i)}_n$ and $\bigcup_{n=1}^\infty A^{(i)}_n$ is dense in $\QQ$, $i=1,2$.
By Prop. 2.2 and Thm. 4.2 of \cite{Dykema-rfd} we deduce tihat $Q\star_{C^*(H)} Q $ is the closure of the union of an increasing sequence of residually finite dimensional $C^*$-subalgebras isomorphic to  $A^{(1)}_n\star_{C^*(H)} A^{(2)}_n$ and hence it is quasidiagonal. One can also invoke
 \cite[Cor.2]{Li-amal} to derive the same conclusion directly.
 Appealing again to \cite[Prop.2.2]{Dykema-rfd}, we deduce that
  \[C^*(G_1)\star_{C^*(H)} C^*(G_2) \subset Q\star_{C^*(H)} Q ,\]
and hence $C^*(G_1\star_{H} G_2)$ is quasidiagonal (if $G_1$ and $G_2$ are amenable).
\end{proof}
\begin{definition}\label{def:wqd}
  A group $G$ is weakly quasidiagonal if there is a ucp asymptotic homomorphism $\{\varphi_n:C^*(G)\to M_{k_n}\}_n$ which separates the points of $G$. In other words this sequence satisfies the conditions  \eqref{ah} and \eqref{sep} from the introduction.
\end{definition}
 Quasidiagonal groups are weakly quasidiagonal. We suspect that the two classes are distinct.
 \begin{proposition}\label{prop:wqdis}
$G$ is weakly quasidiagonal if and only it admits a unitary representation $\pi$ on a separable Hilbert space $H$ for which  there is a sequence  $(p_n)_{n}$ of finite dimensional projections such that $\lim_{n}\|[\pi(s),{p_n}]\|= 0$ and \mbox{$\limsup_{n}\|p_n(\pi(s)-1)p_n\|>0$} for all $s\in G\setminus\{e\}$.
 \end{proposition}
 \begin{proof} In one direction, one observes that the maps $\varphi_n(s)=p_n \pi(s) p_n$ satisfy the conditions  \eqref{ah} and \eqref{sep}.
 Conversely, if $\pi_n:C^*(G) \to L(H_n)$ is the Naimark/Stinespring dilation of $\varphi_n$
   and $p_n$ is the corresponding finite dimensional projection with $\varphi_n=p_n \pi_n p_n$, then $\varphi_n(st)-\varphi_n(s)\varphi_n(t)=p_n\pi_n(s)(1-p_n)\pi_n(t)p_n$ and hence $\|1-\varphi_n(s)\varphi_n(s^{-1})\|=\|p_n\pi_n(s)(1-p_n)\|^2.$  Setting $\pi=\oplus_n \pi_n$, we see that the asymptotic multiplicativity of $\varphi_n$ implies that $\lim_{n}\|[\pi(s),{p_n}]\|= 0$ for all $s\in G$.
 \end{proof}
 \begin{proposition}\label{prop:char-cp-MF}
A group $G$ is  weakly quasidiagonal if and only if there is a  ucp asymptotic homomorphism $\{\psi_n:C^*(G) \to M_{k_n}\}_{n\in \N}$   such that  $\liminf_{n}\|\psi_n(s)-1_{k_n}\|\geq\sqrt{2}$ for all $s\in G\setminus\{e\}$.
\end{proposition}
\begin{proof}
Fix  $s\in G\setminus\{e\}$. It suffices to find a ucp asymptotic homomorphism such that  $\liminf_{n}\|\psi_n(s)-1_{k_n}\|\geq \sqrt{2}$ for this fixed element. Indeed, by  considering direct sums of tails of such sequences,  we can then arrange to have the desired estimate for all nontrivial elements of $G$.
Using the  definition, we first find a ucp asymptotic morphism $\{\varphi_n:C^*(G) \to M_{k_n}\}_{n\in \N}$ such that
$\lim_{n}\|\varphi_n(s)-1_{k_n}\|=\delta>0$.
We will show that there is $m\geq 1$ such that the asymptotic homomorphism  $\psi_n:= \varphi_n ^{\otimes m}$ satisfies $\liminf_{n}\|\psi_n(s)-1_{mk_n}\|\geq \sqrt{2}$. If $\delta \geq \sqrt{2}$, then $m=1$ will do. 
Thus we may assume that
 $\delta=|e^{i\alpha}-1|$ for some $\alpha \in (0,\pi/2).$
By functional calculus there is a sequence of unitaries $u_n\in U(k_n)$ such that $\lim_n \|\varphi_n(s)-u_n\|=0$.
Since $\lim_{n}\|u_n-1_{k_n}\|=|e^{i\alpha}-1|$, $u_n$ has an eigenvalue $e^{i\theta_n}$ with $\|u_n-1_{k_n}\|=|e^{i\theta_n}-1|$ and $\lim_n|\theta_n|= \alpha$. Let $m$ be the smallest integer such that $m\alpha\geq \pi/2$. Then $\lim_n|e^{im\theta_n}-1|=|e^{im\alpha}-1|\geq \sqrt{2}$. Since $e^{im\theta_n}$ is an eigenvalue of $u_n^{\otimes m}$,
it follows that $\liminf_{n}\|\psi_n(s)-1_{mk_n}\|=\liminf_{n}\|u_n^{\otimes m}-1_{mk_n}\|\geq \sqrt{2}.$
  \end{proof}
  \begin{remark} Let $G$ be  a weakly quasidiagonal group.
   If follows from Voiculescu's theorem  that all unital faithful representations {$\pi:C^*(G) \to L(H)$} satisfy the conditions from Proposition~\ref{prop:wqdis}.
 \end{remark}
 \begin{proposition}\label{prop:cp-MF}
An increasing union of  weakly quasidiagonal groups is weakly quasidiagonal.
\end{proposition}
\begin{proof}
Suppose that $G=\bigcup_{n=1}^\infty G_n$, where  $(G_n)_n$ is an increasing sequence of weakly quasidiagonal subgroups of $G$.
Enumerate $G=\{s_n: n\geq 0\}$ with $s_0=e$. We may assume that $F_n=\{s_i: i \leq n\}\subset G_n$.
By Proposition~\ref{prop:char-cp-MF} there is a ucp map $\varphi_n:C^*(G_n) \to M_{k_n}$ such that
$\|\varphi_n(st)-\varphi_n(s)\varphi_n(t)\|<1/n$ for all $s,t \in F_n$ and $\|\varphi_n(s)-1_{k_n}\|\geq \sqrt{2}-1/n$ for all $s\in F_n\setminus \{e\}$. By \cite[Prop.8.8]{Pisier}, the inclusion map $G_n  \subset G$ induces an embedding $C^*(G_n)\subset C^*(G)$ and there is a conditional expectation $E_n:C^*(G)\to C^*(G_n)$ such that $E_n(s)=\chi_{G_n}(s)s$ for all $s \in G$. It is then clear that by setting
$\psi_n:=\varphi_n \circ E_n :C^*(G) \to M_{k_n}$ we obtain an asymptotic homomorphism consisting of ucp maps such that
$\liminf_{n}\|\psi_n(s)-1_{k_n}\|\geq \sqrt{2}$ for all $s\in G\setminus\{e\}$.
  \end{proof}
\begin{example}\label{wqd-not-ra}
 (i) The group $SL_n(\Q)$ ($n\geq 3$)  is simple  by the Jordan-Dickson theorem. It is nonamenable and hence not residually amenable. $SL_n(\Q)$ is weakly quasidiagonal, since as argued in the proof of Corollary~\ref{cor:linear}, all countable linear groups are weakly quasidiagonal.
We don't know if $SL_n(\Q)$ is quasidiagonal.
(ii) Thom \cite{Thom} (cf.Yamashita \cite{Yamashita}) noted that infinite simple property (T) groups are not quasidiagonal,  and in fact  they are not weakly quasidiagonal, see Proposition~\ref{prop:Thom}.
\end{example}
The notion of quasidiagonal trace was introduced in \cite[3.3.1]{Nate:memoirs}. Suppose that the trace $\tau_G$ of $C^*(G)$ induced by the canonical trace of $C^*_r(G)$ is quasidiagonal.
Then $G$ is weakly quasidiagonal. Moreover, since quasidiagonal traces are amenable \cite{Nate:memoirs}, it also follows that $G$ has Kirchberg's factorization property \cite[6.4.2]{Brown-Ozawa}.
The next proposition indicates  a class of quasidiagonal groups $G$ for which $\tau_G$ is quasidiagonal.
\begin{proposition}\label{prop:qdtrace}
If $G$ embeds in the unitary group of a unital $C^*$-algebra $D$ that admits a faithful quasidiagonal trace (e.g. $D$ is a countable product of simple quasidiagonal $C^*$-algebras), then $\tau_G$ is quasidiagonal.
\end{proposition}
\begin{proof}
 This is verified by an argument similar to the proof of the implication $(ii) \Rightarrow (i)$ of Corollary 1.2 in \cite{Kirchberg:factorization}.
 One uses basic permanence properties of quasidiagonal traces established in \cite[3.5]{Nate:memoirs}.
 By assumption, $D$ has a faithful quasidiagonal trace state $\tau$ and there is a unital $*$-homomorphism $\pi:C^*(G) \to D$ which separates the points of $G$.
 Let $s \in G\setminus \{e\}$. Since $\pi(s)\neq 1$ and $\tau$ is faithful, $\tau((\pi(s)-1)^*(\pi(s)-1))=2(1-Re\,\tau(\pi(s)))>0$ and hence $\tau(\pi(s))\neq 1$.
  Define $\rho:C^*(G) \to M_2(D)$ by $\rho(s)=1_D \oplus \pi(s)$.
  Let $\tau_2$ be the tracial state of $M_2(D)$ that extends $\tau$. Then $\tau_2$ is quasidiagonal and $|\tau_2(\rho(s))|<1$ for all $s \in G\setminus \{e\}$.
 The trace $\tau_2^{\otimes n}:M_2(D)^{\otimes n}\to \C$ is quasidiagonal since $\tau_2$ is so, \cite[3.5.7]{Nate:memoirs} .
It follows that for each $n\geq 1$ the trace $tr_n:C^*(G) \to \C$, $tr_n(s)=\tau_2^{\otimes n}(\rho(s)^{\otimes n})=\tau_2(\rho (s))^n$ is quasidiagonal.
Since $lim_n\, tr_n(s)=\delta_e(s)=\tau_G(s)$ for all $s\in G$, we conclude that $\tau_G$ is  quasidiagonal as a weak$^*$-limit of quasidiagonal traces, \cite[3.5.1]{Nate:memoirs}.
\end{proof}

\begin{lemma}\label{lemma:qdmap}
(i) A matricially stable MF-group is MAP. (ii)  If a weakly quasidiagonal  group $G$ is  weakly matricially stable,  then $G$ is MAP.
\begin{proof}
  Both (i) and (ii) are immediate from definitions.
\end{proof}
\end{lemma}
We need the following well-known lemma, see \cite{Ros:qd}, \cite{Thom},  \cite[Thm.4.3.4]{Schneider}.
\begin{lemma}\label{lemma:ros}
Let $\pi:G \to U(H)$ be a unitary representation for which  there is a sequence  $(p_n)_{n}$ of nonzero finite dimensional projections such that $\lim_n\|[\pi(s),{p_n}]\|= 0$ for all $s \in G$.
Then  the trivial representation $\iota_G$ is weakly contained $\pi\otimes \bar{\pi}$.
\end{lemma}
\begin{proof}
For $P\in L(H)$ denote by $\|P\|_2$ its Hilbert-Schmidt norm. Since
\[\|\pi(s)p_n-p_n\pi(s)\|_2\leq \|(\pi(s)p_n-p_n\pi(s))p_n\|_2+\|p_n(\pi(s)p_n-p_n\pi(s))\|_2 \leq 2\|p_n\|_2 \|[\pi(s),{p_n}]\|\]
it follows that
\(\big{\|}\pi(s)\left(\frac{p_n}{\|p_n\|_2}\right)\pi(s)^*-\left(\frac{p_n}{\|p_n\|_2}\right)\big{\|}_2\leq 2\|[\pi(s),{p_n}]\|,\, s \in G.\)
\end{proof}

The following fact was essentially pointed out by Thom on MathOverflow \cite{Thom} .
It shows that an infinite simple property (T) group $G$ is not weakly quasidiagonal.
As explained in \cite[p.93]{Valette-monster} any lattice in $Sp(n,1)$, $n\geq 2$  has uncountably many infinite quotients which are
simple (and torsion).
\begin{proposition}[Ozawa-Thom]\label{prop:Thom}
If   an infinite  property $\mathrm{(T)}$ group $G$ is weakly quasidiagonal,  then $G$ has an infinite residually finite quotient.
\end{proposition}
\begin{proof}
Let $F$ be the set of all equivalence classes of finite dimensional irreducible representations of $G$. If $F$ is infinite, then
$G$ has  an infinite residually finite quotient (whether or not it is weak quasidiagonal).
Let $N=\cap_{\sigma \in F} \ker (\sigma)$.
If $F$ is infinite, then $G/N$ is infinite, since finite groups have only finitely many non-equivalent irreducible representations. Moreover, $G/N$ is MAP since it embeds in $\prod_{\sigma\in F} U(H_\sigma)$. Property {(T)} groups are finitely generated \cite{Bekka-T} and hence $G/N$ is residually finite by Malcev's theorem.

 If the set $F$ is finite, we show that $G$ cannot be weakly quasidiagonal.
  Let $f \in C^*(G)$ be sum of the central Kazhdan projections corresponding to the elements of $F$.
  Then $\|\sigma(f)\|=1$ for any nonzero finite dimensional representation $\sigma$ of $G$.   Write $C^*(G)=fC^*(G)f\oplus (1-f)C^*(G)(1-f)$  (direct
 sum of $C^*$-algebras) where $C=fC^*(G)f$ is finite dimensional.
The canonical morphism $G\to U(C)$ is not injective. Otherwise $G$ would be residually finite and hence finite, since $F$ is finite.
Thus
 there is  $s\in G\setminus\{e\}$ whose image in $U(C)$ is $1$.
 Let $\pi:C^*(G)\to L(H)$ and $(p_n)$ be as in Proposition~\ref{prop:wqdis}. Then $\pi$ decomposes as $\pi=\pi_F \oplus \rho$ where
 $\pi_F=\pi(f)\pi$ and $\rho=(1-\pi(f))\pi$. Since $\lim_{n\to \infty} [\pi(f),p_n]=0$, by functional calculus we find finite dimensional projections $r_n \leq \pi(f)$ and $q_n \leq 1-\pi(f)$ such that $\lim_{n\to \infty}\|p_n-(r_n+q_n)\|=0$.
    Since $\pi_F(s)=1$ as $\pi_F$ factors through $C$, it follows that
 $\limsup_{n\to \infty}\|q_n(\rho(s)-1)q_n\|=\limsup_{n\to \infty}\|p_n(\pi(s)-1)p_n\|>0$.
 Thus $(q_n)$ must have a subsequence $q_{n_k}$ consisting of nonzero projections.
Since $\|[\rho(s),{q_{n_k}}]\|\to 0$ for $s\in G$,  it follows by Lemma~\ref{lemma:ros} that  $\iota_G$ is weakly contained in $\rho\otimes \bar{\rho}$.
  Since $G$ has property {(T)}, it follows that $\iota_G$ is contained in $\rho\otimes \bar{\rho}$, and hence,  by a classic argument \cite[A.1.12]{Bekka-T}, $\rho$ has a nonzero finite dimensional subrepresentation $\sigma$.
  Since $\rho(f)=0$, this implies that $\sigma(f)=0$  which is a contradiction.
\end{proof}
\section{The dual assembly map and quasidiagonality}\label{4}
In this section we prove Theorem~\ref{thm:2} and prepare the ground for Theorem~\ref{thm:af}.

Let $\EG$ be the classifying space for proper actions of $G$, \cite{BCH}. It is known that $\EG$ admits a locally compact model, \cite{Kasparov-Skandalis-Annals}.
Let us recall  that $G$ has a $\gamma$-element if there exists a $G-C_0(\EG)$-algebra $A$ in the sense of Kasparov \cite{Kas:inv} and two elements $d\in KK_G(A,\C)$ and $\eta\in KK_G(\C,A)$ (called Dirac and dual-Dirac elements, respectively) such that the element $\gamma=\eta\otimes_Ad\in KK_G(\C,\C)$ has the property that $p^*(\gamma)=1\in \mathcal{R}KK_G(\EG;C_0(\EG),C_0(\EG))$ where
$p:\EG \to \text{point}$, \cite{Tu:gamma}. We refer the reader to \cite{Kas:inv} for the definitions and the basic properties of these groups and we will freely employ the notation from there.

Let $B$ be a separable $C^*$-algebra endowed with the trivial $G$-action.
Consider the dual assembly map with coefficients in $B$:
$$\alpha: KK_{G}(\C,B) \to RKK_G^0(\underline{E}G;\C,B)$$
defined by $\alpha(y)=p^*(y)$ where $p:\EG \to \text{point}$.
As in \cite{Kas:inv}, we write $RKK_G^0(\underline{E}G;\C,B)$ for ${\CR}KK_G(\underline{E}G;C_0(\EG),C_0(\EG)\otimes B).$
We need the following result of Kasparov \cite[Th.6.5]{Kas:inv},  see also the proofs of
\cite[Thm. 3.1]{Kasparov-Skandalis-kk} and \cite[Thm. 2.3]{Kasparov-Skandalis-cr} for general coefficients.

\begin{theorem}[Kasparov]\label{thm:Kas1}
  If $G$ is a countable discrete group that admits a $\gamma$-element,
  then the dual assembly map $\alpha:KK_G(\C,B)\to {R}KK_G^0(\underline{E}G;\C,B)$ is split surjective with kernel \mbox{$(1-\gamma)KK_G(\C,B)$}.
\end{theorem}
\begin{proof}
Since the two canonical  projections $q_1,q_2:\underline{E}G \times \underline{E}G \to \underline{E}G$  are $G$-homotopic  by \cite[Prop.8]{BCH}, one can simply repeat the arguments from
the proof of \cite[Thm. 6.5]{Kas:inv} (with $X=\underline{E}G$ in place of $EG$).
\end{proof}
  By universality of $\underline{E}G$, there is a $G$-equivariant map (unique up to homotopy) $\sigma: EG \to \underline{E}G$.
  It induces a map $\sigma^*:{R}KK_G^0(\underline{E}G;\C,B)\to {R}KK_G^0({E}G;\C,B)$. Recall that $\QQ=\bigotimes_n M_n$  is universal UHF algebra and $K_0(\mathcal{Q})=\Q$ and $K_1(\mathcal{Q})=0$.  We view $\mathcal{Q}$ as a trivial $G$-algebra.
\begin{corollary}\label{cor:KY}
  Let $G$ be a countable discrete group that admits a $\gamma$-element.
  Then the composition $$\gamma KK_G(\C,\mathcal{Q})\hookrightarrow KK_G(\C,\mathcal{Q})\stackrel{\alpha}\longrightarrow {R}KK_G^0(\underline{E}G;\C,\mathcal{Q})\stackrel{\sigma^*}\longrightarrow {R}KK_G^0({E}G;\C,\mathcal{Q})$$
   is  a surjective map. If moreover $G$ is torsion free,  then $\underline{E}G=EG$ and the composition
  \newline $\gamma KK_G(\C,\mathbb{C})\hookrightarrow KK_G(\C,\C)\stackrel{\alpha}\longrightarrow {R}KK_G^0({E}G;\C,\C)$ is an isomorphism.
\end{corollary}
\begin{proof}
 It was shown in \cite[p.275-6]{BCH} that $\sigma$ induces a rationally injective homomorphism
 \[\sigma_*: RK^G_0(EG) \to RK^G_0(\underline{E}G).\]
   It follows that the map
  \((\sigma_*)^*:\Hom(RK^G_0(\underline{E}G),\Q)\to \Hom( RK^G_0(EG),\Q)\) is surjective.
  By the universal coefficient theorems  stated as Lemma 2.3 of \cite{Kasparov-Skandalis-cr}  and Lemma 3.4 of \cite{Kasparov-Skandalis-kk} applied for the coefficient algebra $\QQ$, the horizontal maps in the commutative diagram
   \[\xymatrix{
{R}KK_G^0(\underline{E}G;\C,\mathcal{Q})\ar[d]_{\sigma^*}\ar[r]& \Hom(RK^G_0(\underline{E}G),\Q)\ar[d]^{(\sigma_*)^*}
\\
{R}KK_G^0({E}G;\C,\mathcal{Q})\ar[r]  & \Hom( RK^G_0(EG),\Q)
}
\]
  are bijections.
  It follows that the restriction map
 $\sigma^*:{R}KK_G^0(\underline{E}G;\C,\mathcal{Q})\to {R}KK_G^0({E}G;\C,\mathcal{Q})$ is surjective. The first part of the statement follows now from Theorem~\ref{thm:Kas1}. The second part follows directly from the same theorem since if $G$ is torsion free then $EG = \underline{E}G$.
\end{proof}
Let $j_G$ and $j_{G,r}$ be the descent maps of Kasparov \cite[Thm.3.11]{Kas:inv}.  Thus $\gamma\in KK_G(\C,\C)$ gives
an element $j_G(\gamma)\in KK(C^*(G),C^*(G))$ which induces a map $$j_G(\gamma)^*=j_G(\gamma) \otimes_{C^*(G)}-:KK(C^*(G),B) \to KK(C^*(G),B).$$
 The image of $j_G(\gamma)^*$ is $j_G(\gamma) \otimes_{C^*(G)} KK(C^*(G),B),$ and as it is customary, we will denote it by $\gamma KK(C^*(G),B)$,
($\gamma_r KK(C_r^*(G),B)$ is defined similarly as the image of $ j_{G,r}(\gamma)^*$).
 Since $G$ is discrete and acts trivially on $B$,
there is a canonical  isomorphism, \cite{Kas:inv},  $$\kappa: KK_G(\C,B)\stackrel{\cong}\longrightarrow KK(C^*(G),B)$$  which is compatible with  the module structure over the group ring of $G$. Thus, by \cite[Lemma 11]{Proietti},
 for every $x\in KK_G(\C,\C)$, the following diagram is commutative.
 \begin{equation}\label{eq:obs}
\xymatrix{
KK_G(\C,B)\ar[r]^-{\kappa}\ar[d]^{x\otimes -}&  KK(C^*(G),B)\ar[d]^{j_G(x)\otimes -}
\\
KK_G(\C,B)\ar[r]_-{\kappa}  & KK(C^*(G),B)
}
\end{equation}
Kasparov \cite[3.12]{Kas:inv} has shown that
  the canonical surjection  $q_G :C^*(G)\to C_r^*(G)$  induces an isomorphism of $\gamma$-parts
   $q_G^*: \gamma_r KK(C_r^*(G),B) \stackrel{\cong}\longrightarrow  \gamma  KK(C^*(G),B)$. In particular:
\begin{proposition}[Kasparov]\label{prop:Kubota}
 If $G$ is a discrete countable group  that admits a $\gamma$-element, then
 $\gamma  KK(C^*(G),B)\subset q_G^* (KK(C_r^*(G),B))$.
\end{proposition}
\begin{proof} We review the argument as it plays an important role in the paper.
Let $A$, $d$ and $\eta$ be as in the definition of the $\gamma$-element.
The following diagram is commutative  in the category with objects separable $C^*$-algebras and morphisms KK-elements \cite{Kas:inv}.
\[
\xymatrix{
C^*(G)\ar[r]^{j_G(\eta)}\ar[d]^{q_G}& A\rtimes G\ar[d]^{q_A}\ar[r]^{j_G(d)}& C^*(G)\ar[d]^{q_G}
\\
C_r^*(G)\ar[r]_{j_{G,r}(\eta)}  &A\rtimes_r G \ar[r]_{j_{G,r}(d)}  & C_r^*(G)
}
\]
Since the action of $G$ is proper, the natural map $q_A:A\rtimes G\to A\rtimes_r G$ is an isomorphism by \cite[3.4.16]{Echterhoff:book}. We obtain then a commutative diagram of abelian groups
\begin{equation*}
\xymatrix{
KK(C^*(G),B)& KK(A\rtimes G,B)\ar[l]_-{j_G(\eta)^*}\ar@{=}[d]& KK(C^*(G),B)\ar[l]_{j_G(d)^*}
\\
KK(C_r^*(G),B)\ar[u]^{q_G^*}  &KK(A\rtimes_r G,B)\ar[l]^-{j_{G,r}(\eta)^*}  &
}
\end{equation*}
which shows that $\mathrm{Im}(j_G(\gamma)^*)=\mathrm{Im}(j_G(\eta)^*\circ j_G(d)^* ) \subset \mathrm{Im}(j_G(\eta)^*)\subset \mathrm{Im}(q_G^*)$.  Note that for $B=\C$ this shows that $\gamma\in K^0(C^*(G))$ is in the image of $q_G^*$ and this means that it can be represented by a class whose underlying $G$-representation is weakly contained in the left regular representation. This property of $\gamma$ is part of the hypotheses of \cite[3.12]{Kas:inv}.
\end{proof}

Let $A$, $B$ be separable $C^*$-algebras.
Any class $x\in KK(A,B)$ is  represented by some Cuntz pair, i.e. a pair
   of $*$-homomorphisms $\varphi,\psi:A \to M(K(H)\otimes B)$,  such that $\varphi(a)-\psi(a)\in K(H)\otimes B$, for all $a\in A$.
The following notion was introduced in \cite{AA} for K-homology classes. Assume that $B$ is unital.
\begin{definition}\label{def:kkqd}
An element $x\in KK(A,B)$ is quasidiagonal if it is represented by a Cuntz pair  $\varphi,\psi:A \to M(K(H)\otimes B),$ $\varphi(a)-\psi(a) \in K(H)\otimes B$, for all $a\in A,$  with the property that
there exists an approximate unit of projections $(p_n)_n$ of  $K(H)$ such that $\lim_{n\to \infty}\|[\psi(a),p_n\otimes 1_B]\|= 0$,  for all $a\in A$.
The quasidiagonal elements form a subgroup of $KK(A,B)$,  denoted by $KK(A,B)_{qd}$. The $C^*$-algebra $A$ is called K-quasidiagonal if $K^0(A)=K^0(A)_{qd}$.
  \end{definition}
  \begin{remark}\label{remark:qdd}
  (a) If $\theta:A \to D$ is a $*$-homomorphism, then $\theta^*[\varphi,\psi]= [\varphi\circ \theta,\psi\circ \theta]$ and hence $$\theta^*(KK(D,B)_{qd})\subset KK(A,B)_{qd}.$$

    (b) Let $D, B$ be  separable $C^*$-algebras with $B$ nuclear and unital. Fix a faithful representation $\psi_0:D \to M(K(H))$ such that $\psi_0(D)\cap K(H)=\{0\}$ and $(\psi_0(D)H)^\perp$ is infinite dimensional. Then any element $x\in KK(D,B)$ is represented by a Cuntz pair $(\varphi,\psi)$ where $\psi=\psi_0 \otimes 1_B: A \to M(K(H)\otimes B)$, \cite{Skandalis:K-nuclear}. Therefore, if  $D$ is quasidiagonal,
     then $ \psi_0(D)$ is a quasidiagonal subset of $M(K(H))$ and hence
      $KK(D,B)=KK(D,B)_{qd}$.

  \end{remark}
By \cite[Thm.3.4]{Kas:inv}, there is natural descent isomorphism
  $$ \lambda^G:{R}KK_G^0({E}G;\C,B)\stackrel{\cong}\longrightarrow  {R}KK^0(BG;\C,B).$$
In particular,
$ {R}KK_G^0({E}G;\C,\mathcal{Q})\cong  {R}KK^0(BG;\C,\mathcal{Q})=: RK^0(BG;\Q)$, \cite{Kasparov:conspectus}. Let $\nu$ be the map
\begin{equation*}
 \nu=\lambda^G\circ\sigma^*\circ \alpha \circ \kappa^{-1} : KK^0(C^*(G),B)\cong KK_G(\C,B) \to {R}KK_G^0({E}G;\C,B)\cong  {R}KK^0(BG;\C,B).
\end{equation*} Note that $\sigma^*\circ \alpha =p^*$ where $p:EG \to point$.

 Kubota's idea \cite{Kubota2} of using a quasidiagonal $C^*$-algebra intermediate between $C_r^*(G)$ and $C^*(G)$
has strengthened significantly the approach to  almost flat K-theory based on {$K$}-{quasidiagonality} of $C^*(G)$,  introduced in \cite{AA}.  By Proposition~\ref{cor:qd-embed}, the quasidiagonal groups are exactly those which admit quasidiagonal intermediate $C^*$-algebras. This enables us to extend a result from \cite{Kubota2} to the class of quasidiagonal groups
(we consider only the non-relative, single group case).  Moreover, the direct use of Theorem~\ref{thm:Kas1} allows for integral coefficients in the case of torsion free groups.
\begin{theorem}\label{cor:qd-embed-kk}
  Let $G$ be a countable discrete quasidiagonal group and let $B$ be a separable nuclear unital $C^*$-algebra.
  If $G$ admits a $\gamma$-element, then 
$ \gamma KK(C^*(G),B) )\subset KK(C^*(G),B)_{qd}$. It follows that
  \(\nu(KK(C^*(G),\mathcal{Q})_{qd})= RK^0(BG;\Q) \) and if we also assume that $G$ is torsion free, then \(\nu(K^0(C^*(G))_{qd})= RK^0(BG)\).
\end{theorem}
\begin{proof} The factorization
$C^*(G) \stackrel{q_D}\longrightarrow D  \to C_r^*(G)$ of $q_G$
with $D$ unital and quasidiagonal given by Proposition~\ref{cor:qd-embed} in conjunction with Remark~\ref{remark:qdd} implies that
\[q_G^* (KK(C_r^*(G),B))\subset q_D^*(KK(D,B))= q_D^*(KK(D,B)_{qd})\subset KK(C^*(G),B)_{qd}.\]
From this and Proposition~\ref{prop:Kubota} we obtain that $ \gamma KK(C^*(G),B) )\subset KK(C^*(G),B)_{qd}$.

By  Theorem~\ref{thm:Kas1},  $\alpha$ vanishes on $(1-\gamma) KK_G(\C,B))$. Since the diagram \eqref{eq:obs} is commutative, this group is mapped to  $(1-\gamma) KK^0(C^*(G),B)$ by $\kappa$   and hence
$$\nu (KK^0(C^*(G),B))= \nu(\gamma KK^0(C^*(G),B))= \nu(KK(C^*(G),B)_{qd}). $$
   By  Corollary~\ref{cor:KY}, the map $\nu$ is surjective if either $B=\QQ$ or if $G$ is torsion free and $B=\C$.
\end{proof}
\begin{remark}\label{rem:push}
 Suppose that $\{\pi_n:A \to D_n\}_n$ is a ucp asymptotic homomorphism
of unital C*-algebras. Thus $\lim_{n \to \infty} \|\pi_n(aa')-\pi_n(a)\pi_n(a')\|=0$ for all $a,a'\in A$. The sequence $\{\pi_n\}_{n}$ induces a unital $*$-homomorphism
$A \to \prod_n D_n/\bigoplus_n D_n$ and hence a group homomorphism
$K_0(A)\to \prod_n K_0(D_n)/\bigoplus_n K_0(D_n)$.
This gives a canonical way to push forward  an element $x\in K_0(A)$ to a sequence $(\pi_{\nsh }(x))_n$ with components in $K_0(D_n)$ which is well-defined up to tail equivalence: two sequences are tail equivalent,
$(y_n)\equiv (z_n)$, if there is $m$ such that $x_n=y_n$ for all $n\geq m$. Note that $\pi_{\nsh }(x+x')\equiv \pi_{\nsh }(x)+\pi_{\nsh }(x')$.
Of course, if $\pi_n$ are genuine $*$-homomorphisms then $\pi_{\nsh }(x)=\pi_{n\, * }(x)$.
It is convenient to extend the notation above for the corresponding maps on $K_1$ and K-theory with coefficients $K_*(A;\Z/p)\to \prod_n K_*(D_n;\Z/p)/\bigoplus_n K_*(D_n;\Z/p)$.
\end{remark}

 If $BG$ is written as the union of an increasing sequence
$(Y_i)_{i}$ of finite CW complexes, then as explained in the proof of Lemma 3.4 from \cite{Kasparov-Skandalis-kk}, there is a Milnor $\varprojlim^1$ exact sequence which gives
\begin{equation}\label{eq:milnor}
RK^0(BG;\Q)\cong \varprojlim RK^0(Y_i;\Q).
\end{equation}
We denote by $\nu_i$ the composition of the map $\nu$ defined above with the restriction map $RK^0(BG;\Q)\to RK^0(Y_i;\Q)$. If $Y=Y_i$ is fixed we will write $\nu_Y$ for $\nu_i$.
\vskip 4pt
\emph{Proof of Theorem~\ref{thm:2}:}
\vskip 4pt
 If $G$ is weakly quasidiagonal and weakly matricially stable, then by Lemma~\ref{lemma:qdmap}, $G$ is MAP and in particular quasidiagonal. It follows  by Theorem~\ref{cor:qd-embed-kk} that
 \(\nu(KK(C^*(G),\mathcal{Q})_{qd})= RK^0(BG;\Q) \).

 Consider the flat line-bundle $\ell$ with fiber $C^*(G)$ defined by
$\widetilde{EG}\times_G C^*(G)\to BG$, where $G\subset C^*(G)$ acts diagonally.
  Let $Y$ be a  connected finite CW complex $Y\subset BG$. The restriction of $\ell$ to $Y$, denoted $\ell_Y,$ yields a self-adjoint projection $P=P_Y$
in $M_m(\CCC)\otimes C(Y)\otimes C^*(G)$ for some $m\geq 1$ and hence an element $[P]\in KK(\C,C(Y)\otimes C^*(G))$. We shall elaborate on this point shortly.

It was shown by Kasparov \cite[Lemma~6.2]{Kas:inv}, \cite{Kasparov:conspectus},
that the map \[\nu_Y:KK(C^*(G),\QQ) \stackrel{\nu}\longrightarrow K^0(BG;\Q) \to K^0(Y;\Q) \cong K_0(C(Y)\otimes \QQ)\] is given by $\nu_Y(x)=[P]\otimes_{C^*(G)} x$.
Let  $(U_i)_{i\in I}$ be finite covering of $Y$ by open sets
such that $\ell$ is trivial on each $U_i$ and $U_i\cap U_j$ is connected. Using  trivializations
of $\ell$  to $U_i$ one obtains group elements $s_{ij}\in G$
 which define  a 1-cocycle that is constant on each nonempty set $U_i\cap U_j$  and which represents $\ell_Y$.  Thus $s_{ij}^{-1}=s_{ji}$ and $s_{ij} \cdot s_{jk}=s_{ik}$ whenever $U_i\cap U_j \cap U_k \neq \emptyset.$
 Let $(\chi_i)_{i\in I}$
be positive continuous functions with $\chi_i$ supported in $U_i$ and such that
$\sum_{i\in I} \chi^2_i=1$. Set $m=|I|$ and let $(e_{ij})$ be the canonical matrix unit of $M_m(\CCC)$. Then $\ell_Y$ is represented by the selfadjoint projection
\[P=\sum_{i,j\in I} e_{ij}\otimes \chi_i\chi_j\otimes s_{ij}\in M_{m}(\CCC)\otimes C(Y)\otimes C^*(G).\]
We now take advantage of the following realization of $\nu_Y$ on quasidiagonal KK-classes introduced in  \cite{AA}.
Let $x\in KK(C^*(G),\QQ)_{qd}$. Then  $x$ is represented  by  a pair of nonzero $*$-representations $\varphi,\psi:C^*(G)\to M(K(H)\otimes \QQ)$,  such that $\varphi(a)-\psi(a)\in K(H)\otimes \QQ$,
 $a\in C^*(G)$, and with the property that
 there is an increasing approximate unit $(p_n)_n$ of $K(H)$ consisting of projections such that $(p_n\otimes 1_\QQ)_n$ commutes asymptotically with both $\varphi(a)$ and $\psi(a)$, for all $a\in C^*(G)$.
  It is then clear that the compressions $\varphi^{(0)}_n=(p_n\otimes 1_\QQ) \varphi(\cdot)(p_n\otimes 1_\QQ)$ and $\varphi^{(1)}_n=(p_n\otimes 1_\QQ) \psi(\cdot)(p_n\otimes 1_\QQ)$
 are cp asymptotic homomorphisms $\varphi^{(r)}_n:C^*(G)\to K(H)\otimes \QQ$. Let $1$ denote the unit of $C^*(G)$.
It is routine to  further perturb these maps to cp asymptotic homomorphisms such that
$\varphi^{(r)}_n(1)$, $r=0,1$, are projections. If $[\varphi^{(r)}_n(1)]= k^{(r)}_n/q^{(r)}_n \in \Q=K_0(\QQ)$, then after conjugating by  unitaries in $\QQ$, we can arrange that $\varphi^{(r)}_n:C^*(G)\to M_{k^{(r)}_n}(\C e_{0})\subset M_{k^{(r)}_n}(\QQ^{(r)}_n)$, where
 $\QQ^{(r)}_n=(1 \otimes \cdots \otimes 1\otimes  M_{q^{(r)}_n}\otimes 1\cdots )  \subset \bigotimes_{q=1}^\infty M_q=\Q,$
and $e_{0}$ is a rank one projection in a copy of $M_{q^{(r)}_n}\cong\QQ^{(r)}_n$.
As argued in  \cite[Prop.2.5]{AA}: 
 \begin{equation}\label{eqn:important}
  \nu_Y(x)=[P]\otimes_{C^*(G)} x \equiv (\id{m}\otimes\id{C(Y)}\otimes\varphi^{(0)}_n)_\sharp(P)-( \id{m}\otimes\id{C(Y)}\otimes\varphi^{(1)}_n)_\sharp(P).
 \end{equation}
 Since $G$ is assumed to be weakly matricially stable, there exist sequences of genuine group representations $\sigma^{(r)}_n$ and $\pi^{(r)}_n$, $r=0,1$, such that
\begin{equation}\label{eqn:importanti}\lim_n\|\varphi^{(r)}_n(s)\oplus \sigma^{(r)}_n(s)-\pi^{(r)}_n(s)\|=0 \end{equation} for all $s\in G$.
  It is clear that we can view $\sigma^{(r)}_n$ and $\pi^{(r)}_n$ as maps into matrices over $\QQ^{(r)}_n$.
  Let $x_n$ denote the class of $[\sigma^{(0)}_n,\sigma^{(1)}_n]\in KK(C^*(G),\QQ).$
  Then
 \begin{equation}\label{eqn:importantii}\nu_Y(x_n)= [P]\otimes_{C^*(G)} (x_n)\equiv (\id{m}\otimes\id{C(Y)}\otimes\sigma^{(0)}_n)_*[P]-( \id{m}\otimes\id{C(Y)}\otimes\sigma^{(1)}_n)_*[P].\end{equation}
It follows from \eqref{eqn:important}, \eqref{eqn:importanti} and \eqref{eqn:importantii} that
\begin{equation}\label{eqn:iimportant}
 \nu_Y(x+x_n)= [P]\otimes_{C^*(G)} (x +x_n)\equiv (\id{m}\otimes\id{C(Y)}\otimes\pi^{(0)}_n)_*[P]-( \id{m}\otimes\id{C(Y)}\otimes\pi^{(1)}_n)_*[P].
 \end{equation}
 On the other hand we observe that the projections $e^{(r)}_n$, $f^{(r)}_n$, $r=0,1$, defined by
 \[e^{(r)}_n=(\id{m}\otimes\id{C(Y)}\otimes\pi^{(r)}_n)(P)=\sum_{i,j\in I} e_{ij}\otimes \chi_i\chi_j\otimes \pi^{(r)}_n(s_{ij})\in M_{m}(\CCC)\otimes C(Y)\otimes M_{K(n)}(\QQ^{(r)}_n),\]
  \[f^{(r)}_n=(\id{m}\otimes\id{C(Y)}\otimes\sigma^{(r)}_n)(P)=\sum_{i,j\in I} e_{ij}\otimes \chi_i\chi_j\otimes \sigma^{(r)}_n(s_{ij})\in M_{m}(\CCC)\otimes C(Y)\otimes M_{K(n)}(\QQ^{(r)}_n),\]
 correspond to flat finite rank complex bundles since they are realized via the constant cocycles $\pi^{(r)}_n(s_{ij})$ and $\sigma^{(r)}_n(s_{ij})$.
 By an extension of a result of Milnor given in \cite{KT:flat-bundles}, all the rational Chern classes of flat complex bundles vanish. It follows that the bundles corresponding to $m_n\cdot e^{(r)}_n$ and $m_n\cdot f^{(r)}_n$ will be trivial for suitable integers $m_n$ and hence $$\nu_Y(x)=\nu_Y(x+x_n)-\nu_Y(x_n)=[e^{(0)}_n]-[e^{(1)}_n]-[f^{(0)}_n]+[f^{(1)}_n]\in {K}^0(point;\Q).$$ Since $x$ was arbitrary, it follows that
 $\nu(KK(C^*(G),\QQ)_{qd})\subset \Q =RK^0(point;\Q)\subset RK^0(BG;\Q)$.
On the other hand we know that $\nu$ is surjective by Theorem~\ref{cor:qd-embed-kk}. Thus $RK^0(BG;\Q)=\Q$ and hence
  $H^{even}(BG;\Q)=\Q,$ since the Chern character
\( Ch: RK^0(BG;\Q)\to H^{even}(BG;\Q)\)
is an isomorphism.
 \qed

\begin{remark}\label{rem:flat}
(i) Let $G$ be as in Theorem~\ref{thm:2} and assume in addition that $BG$ admits a finite simplicial complex model.
If  $G$ is weakly matricially stable it follows that $K^0(BG)$ is generated by flat bundles.
Indeed, if $BG$ is a finite simplicial complex, then $G$ is torsion free and so $\nu:K^0(C^*(G))_{qd}\to K^0(BG)$ is surjective by Theorem~\ref{cor:qd-embed-kk}.
Given $y\in  K^0(BG)$ we lift $y$ to quasidiagonal class $x=[\varphi,\psi]$ and reasoning as above we see that $y=[e^{(0)}_n]-[e^{(1)}_n]-[f^{(0)}_n]+[f^{(1)}_n]$
for sufficiently large $n$.

(ii) If $G$ is a group with Haagerup's property and $BG$ admits a finite simplicial complex model, then we claim that $K^0(C^*(G))=R(G)_{fin}$ if and only if
$K^0(BG)$ is generated by flat bundles.
This is explained by the fact that $\nu$ is a generalization of the Atiyah-Segal map and that it is a bijection under the present assumptions.
 If $\rho: G \to U(k)$ is a unitary representation,
then $\nu[\rho]=[E_\rho]$ where $E_\rho=EG\times_G \C^k$ ($G$ acts on $\C^k$ via $\rho$) is a flat hermitian bundle.
Conversely, if $E$ is a flat hermitian bundle of rank $k$ over $BG$, it is well-known  that $E\cong E_\rho$ where $\rho:\pi_1(G)\cong G \to U(k)$ is the monodromy representation of $E$, see (for example )\cite{DadCarrion:almost_flat}.
\end{remark}
\section{Matricially stable $C^*$-algebras}\label{sec:msa}
In this section we prove Theorems~\ref{thm:3} and \ref{thm:amen} by establishing first their $C^*$-algebraic versions.
\begin{definition}\label{def:alg-wms}
A unital $C^*$-algebra $A$ is weakly matricially stable if for any ucp asymptotic homomorphism $\{\varphi_n:A \to M_{k_n}\}_n$ there are two sequences of unital finite dimensional $*$-representations  of $A$, denoted $(\pi_n)_n$ and $(\sigma_n)_n$, such that $\lim_n\|\varphi_n(a)\oplus\sigma_n(a) -\pi_n(a)\|=0$ for all $a\in A$. Note that a group $G$ is weakly matricially stable if and only if $C^*(G)$ is so.
\end{definition}
Recall that $\mathrm{Inf} K_0(A)=\{x\in K_0(A)\,:\,  \forall k \in \Z, \, \exists m>0,\, m([1_A]+kx )\geq 0\}$
is the subgroup of infinitesimal elements of $K_0(A)$. If $\varphi:A \to B$ is a unital $*$-homomorphism, then $\varphi_*(\mathrm{Inf} K_0(A))\subset \mathrm{Inf} K_0(B)$.
\begin{theorem}\label{thm-gen}
 Let $A$ be a separable unital K-quasidiagonal $C^*$-algebra that satisfies the UCT and such that $K_*(A)$ is finitely generated.
 Suppose that $A$ is weakly matricially stable.
 Then $K^0(A)$ is generated by finite dimensional  representations of $A$ and $\mathrm{Inf} K_0(A)$ is
 a torsion group.
\end{theorem}
\begin{proof}
 $A$ satisfies the UCT if and only if
  it satisfies the UMCT of \cite{DadLor:duke}.  A concise account of the UMCT is included in \cite{Dad-Meyer}.
  Let \(\mathcal{P}\subseteq\N\) be the set consisting of~\(0\) and all prime powers.
   The total K-theory group of a $C^*$-algebra $A$ is defined as
 \(
\uK(A)=\bigoplus_{p\in \mathcal{P}} K_*(A;\Z/p).
\)
$\uK(A)$ is a module over the ring $\Lambda$ of {B\"ockstein operations}.
  Since  $K_1(A)$ is finitely generated, $\mathrm{Pext}(K_1(A),K_0(\C))=0$ and by the UMCT we have an isomorphism
\begin{equation}\label{umct}
\theta: K^0(A)\stackrel\cong\longrightarrow \Hom_{\Lambda}(\uK(A),\uK(\C))
\end{equation}
induced by the Kasparov product  $K_*(A;\Z/p)\times K^0(A) \to K_*(\C;\Z/p),$ $\theta(x)(y)=y\otimes_A x$.
Since $K_*(A)$ is finitely generated, $\uK(A)$ is a finitely generated $\Lambda$-module. Moreover, if  $\mathrm{Tor} K_*(A)$ has order $N$ and
$\mathcal{P}_N=\{p\in \mathcal{P}: p | N\}\cup \{0\}$, then by \cite[Cor.2.11]{DadLor:duke} we can replace $\uK(-)$ and $\Lambda$ in \eqref{umct} by $\uK'(-)$ and $\Lambda'$, where $\uK'(A)=\bigoplus_{p\in \mathcal{P}_N} K_*(A;\Z/p)$ and $\Lambda'$ is the corresponding subset of {B\"ockstein operations} acting on $\uK'(A)$.
Since $A$ is K-quasidiagonal, any element $x\in K^0(A)$ is represented by a Cuntz pair $(\varphi^{(0)},\varphi^{(1)})$ for which there is an increasing approximate unit of $K(H)$ consisting of projections, denoted $(p_n)_n$, such that $\lim_{n}\|[\varphi^{(i)}(a),p_n]\|=0$, $i=0,1$, for all $a\in A$.
Then $p_n\varphi^{(i)}(\cdot)p_n$ are finite rank cp asymptotic morphisms which we can perturb to ucp asymptotic homomorphisms $\varphi_n^{(i)}:A \to M_{k_n^{(i)}}$.
It follows  by \cite[Prop.2.5]{AA} that for each $y\in \uK(A):$
\[(\varphi^{(0)}_{n})_\sharp(y)-(\varphi^{(1)}_{n})_\sharp(y)\equiv y\otimes_{A} x.\]
Suppose now that $A$ is weakly matricially stable. Then there are sequences of finite dimensional representations of $A$, $(\sigma^{(i)}_n)$ and  $(\pi^{(i)}_n)_n$, $i=0,1$, such that
$\lim_n\|\varphi^{(i)}_{n}(a)\oplus \sigma^{(i)}_{n}(a)-\pi^{(i)}_{n}(a)\|=0$, $i=0,1$, for all $a \in A$.
  Let $x_n$ denote the class of $[\sigma^{(0)}_n,\sigma^{(1)}_n]\in K^0(A).$
Since K-theory is stable under small perturbations, we obtain that for all $y\in \uK(A):$
\[(\pi^{(0)}_{n})_*(y)-(\pi^{(1)}_{n})_*(y)\equiv (\varphi^{(0)}_{n}\oplus \sigma^{(0)}_{n})_\sharp (x) -(\varphi^{(1)}_{n}\oplus \sigma^{(1)}_{n})_\sharp (x)\equiv y\otimes_{A} (x+x_n).\]
Since the maps $\pi^{(i)}_{n\, *}$  are $\Lambda'$-linear and since the $\Lambda'$-module $\uK'(A)$ is finitely generated, there is a sufficiently large  $m$ such that $(\pi^{(0)}_{m})_*(y)-(\pi^{(1)}_{m})_*(y)=y\otimes_{A}( x+x_m)=y\otimes_{A} x+(\sigma^{(0)}_{m})_*(y)-(\sigma^{(1)}_{m})_*(y)$ for all $y \in  \uK'(A)$.
Since $\theta$ is an isomorphism, it follows that  $x=[\pi^{(0)}_{m}]-[\pi^{(1)}_{m}]-[\sigma^{(0)}_{m}]+[\sigma^{(1)}_{m}]$. This finishes the proof of the first part of the statement.

We prove the second part by contradiction. If $y\in \mathrm{Inf} K_0(A)$ is a non-torsion element, then there is a homomorphism $h:K_0(A)\to \Z$ such that $h(y)\neq 0$.
By the UCT, $\theta$ induces a surjection $K^0(A)\to Hom(K_0(A),\Z)$.
 Thus there is $x\in K^0(A)$ such that $\theta(x)=h$. By the first part of the proof, we find two finite dimensional representations $\pi_i:A \to M_{k_i}$ of $A$ such that $x=[\pi_0]-[\pi_1].$ It follows that $h(y)=\theta(x)(y)=(\pi_0)_*(y)-(\pi_1)_*(y)=0$ since
 $\mathrm{Inf}K_0(\C)=0$.
\end{proof}

 Eilers, Loring and Pedersen \cite{ELP} introduced and studied natural noncommutative analogues of CW complexes (NCCW). They showed that if $A$ is a
 2-dimensional NCCW such that $\mathrm{Inf} K_0(A)$ is a torsion group, then $A$ is matricially stable.
 We show below that the reverse implication also holds.
 The $C^*$-algebras associated with wallpaper groups are examples of 2-dimensional NCCW complexes.
 The list of all matricially stable crystallographic groups was obtained \cite{ESS}. In addition to Theorem~\ref{thm:2}, the following corollary sheds new lights on the results of \cite{ESS} concerning these groups, see \cite[Remarks 4.4-4.5]{ESS}.
\begin{corollary}
  A 2-dimensional NCCW $A$ is matricially stable if and only $\mathrm{Inf} K_0(A)$ is a torsion group.
\end{corollary}
\begin{proof}
  The implication $(\Leftarrow)$ was proven in \cite[Cor.8.2.2]{ELP}. The converse follows from Theorem~\ref{thm-gen} since $A$
  is type I, it is quasidiagonal and it has finitely generated K-theory.
\end{proof}
\begin{theorem}\label{thm:amen-alg}
  Let $A$ be a separable unital exact quasidiagonal $C^*$-algebra satisfying the UCT and such that $K_*(A)$ is finitely generated.
Then $A$ is weakly matricially stable if and only if $A$ is residually finite dimensional and $K^0(A)$  is generated by finite dimensional  representations of $A$.
\end{theorem}
\begin{proof}
  $(\Rightarrow)$ Since $A$ is quasidiagonal, there is a ucp asymptotic homomorphism $\{\varphi_n:A \to M_{k_n}\}_n$ which is asymptotically isometric. If $(\sigma_n)_n$ and $(\pi_n)_n$ are as in Definition~\ref{def:alg-wms}, then $(\pi_n)$ is also asymptotically isometric and hence $A$ is residually finite dimensional.
  The required property of $K^0(A)$ follows from Theorem~\ref{thm-gen}. $(\Leftarrow)$
  In order to prove that $A$ is weakly matricially stable is suffices to show that
  for any ucp asymptotic homomorphism $\{\varphi_n:A \to M_{k_n}\}_n$, for any finite subset $F$ of $A$ and $\varepsilon >0$,
  there are two sequences of finite dimensional representations, $(\alpha_n)$ and $(\beta_n)$, such that
\begin{equation}\label{eqn:amb}
  \limsup_n\|\varphi_n(a)\oplus \alpha_n (a)- \beta_n(a) \|<\varepsilon.
\end{equation}
  for all $a\in F$.
  Since $K_*(A)$ is finitely generated, as explained in the proof of Theorem~\ref{thm-gen}, we can reduce the $\Lambda$-module $\uK(A)$ in the UMCT to a finite array of groups and a finite number of B\"ockstein operations denoted by $\uK'(A)$ and $\Lambda'$. Thus the UMCT reads: $KK(A,B)\cong \Hom_{\Lambda'}(\uK'(A),\uK'(B))$.
     Since $\uK'(A)$ is a finitely generated $\Lambda'$ module and since the set $\Lambda'$ is finite we can associate to each $\varphi_n$ an element $h_n\in \Hom_{\Lambda'}(\uK'(A),\uK'(M_{r_n}))$
    such that $(\varphi_n)_\sharp(x) \equiv h_n(x)$, for all $x\in \uK'(A)$. Since $K^0(A)\cong \Hom_{\Lambda'}(\uK'(A),\uK'(\C))$, it follows from our assumption on $K^0(A)$ that for each $n$ there are finite dimensional representations $\sigma_n$ and $\psi_n$ of $A$ such that $h_n=(\psi_n)_*-(\sigma_n)_*$.  Thus,
  \((\varphi_n)_\sharp(x) \equiv(\psi_n)_*(x)-(\sigma_n)_*(x)\), for all $x\in \uK'(A)$.  Setting $\phi_n=\varphi_n \oplus \sigma_n:A \to M_{r_n}$ it follows that $( \phi_n)_\sharp (x)\equiv (\psi_n)_\sharp(x)$  for all  $x\in \uK'(A)$. Let $D=\prod_n M_{r_n}$, $I=\bigoplus_n M_{r_n}$ and $B=D/I$.
Consider the  unital $*$-homomorphisms $\Phi, \Psi :A \to B$ induced by the asymptotic homomorphisms $(\phi_n )_n$ and $(\psi_n)_n$. One verifies that $\Phi_*= \Psi _*: \uK'(A) \to \uK'(B)$ and hence that $[\Phi]=[\Psi]$ in $KK(A,B)$.
We substantiate this claim as follows.
Let $B_n=\prod_{k\geq n} M_{r_k}$. Then $B=\varinjlim B_n$ where $B_n \to B_{n+1}$ is the natural restriction map.
We are going to show that the map \[J: \uK(B) \cong \varinjlim \uK(\prod_{k\geq n} M_{r_k})\to \varinjlim \prod_{k\geq n} \uK(M_{r_k})\cong \prod_n \uK(M_{r_n})/\bigoplus_n \uK(M_{r_n})\] is injective. This will follow if one shows
 that the natural map $j_n: \uK(\prod_{k\geq n} M_{r_k})\to \prod_{k\geq n}\uK( M_{r_k})$ is injective. To that purpose we note that $K_1(B_n)=0$ and that
 \[K_0(B_n)=\{(x_k)_{k\geq n}\,:\, x_k \in \Z,\, \exists m\, \text{with}\, |x_k|\leq m r_k, \forall k\geq n \}.\]
 Since the map $K_0(B_n)\stackrel{\times p}\longrightarrow K_0(B_n)$ is injective, we deduce that $K_1(B_n;\Z/p)=0$
for $p\geq 2$ and all $n$. Using the exactness and the naturality of the B\"ockstein  sequence,  we have the following commutative diagram which is easily checked to have injective vertical maps:
$$
\xymatrix{
{0} \ar[r] & {K_0(\prod_{k\geq n} M_{r_k})}\ar[r]^{\times p}\ar[d] & {K_0(\prod_{k\geq n} M_{r_k})}\ar[r]\ar[d] &{K_0(\prod_{k\geq n} M_{r_k};\Z/p)}\ar[r]\ar[d]& {0}\\
{0} \ar[r] & {\prod_{k\geq n} K_0( M_{r_k})}\ar[r]^{\times p} & {\prod_{k\geq n} K_0( M_{r_k})}\ar[r] &{\prod_{k\geq n} K_0( M_{r_k};\Z/p)}\ar[r]& {0}
}
$$
Therefore all         the maps $j_n$ are injective and hence $J$ is injective. We have seen earlier  that  $( \phi_n)_\sharp (x)\equiv (\psi_n)_\sharp(x)$ for all $x\in \uK(A)$. This means precisely that $J \circ \Phi_*=J \circ \Psi_*$. Since $J$ is injective, it follows that $\Phi_*= \Psi_*$ as desired.
Let $(\eta_k)_k$ be a sequence of finite dimensional representations of $A$ that separates the points of $A$ and such that each $\eta_k$ occurs infinitely many times. Set $\gamma_k =\eta_1\oplus \cdots \oplus \eta_k$ and view it as a map
  $\gamma_k:A \to M_{m_k}(\C 1_B)\subset M_{m_k}(B)$. Let us note that $A$ is K-nuclear since it is KK-equivalent to a commutative $C^*$-algebra and hence $KK(A,B)\cong KK_{nuc}(A,B)$. Moreover, since $A$ is exact, the maps $\Phi$ and $\Psi$ are nuclear by \cite[Prop.3.3]{Dad;qdmor}.
  By \cite[Thm.4.3]{DadEil:AKK}, since $[\Phi]=[\Psi]$ in $KK(A,B)$, there is a sequence of unitaries $u_k\in M_{m_k+1}(B)$ such that $\lim_n\|\Psi(a)\oplus \gamma_k(a)-u_k(\Psi(a)\oplus \gamma_k(a))u_k^*\|=0$ for all $a\in A$.
 Thus, for any finite subset $F$ of $A$ and $\varepsilon >0$ there is a $k$ such that if we set $\gamma:=\gamma_k$, $m:=m_k$, and $u:=u_k$, then $\|\Psi(a)\oplus \gamma(a)-u(\Psi(a)\oplus \gamma(a))u^*\|<\varepsilon$ for all $a\in F$.
   We can view $\gamma$ as a map into $M_m(\C 1_D)$ with components $(\gamma_n)$. Lift  $u$ to a unitary $v\in M_{m+1}(D)$ with components $(v_n)$.
   It follows that for all $a\in F,$
   \begin{equation}\label{eqn:ama}
  \limsup_n\|\varphi_n(a)\oplus \sigma_n(a)\oplus \gamma_n (a)-v_n(\psi _n(a)\oplus \gamma_n (a))v_n^*\|<\varepsilon.
\end{equation}
This proves \eqref{eqn:amb} with $\alpha_n=\sigma_n\oplus \gamma_n$ and $\beta_n=v_n(\psi _n\oplus \gamma_n)v_n^*$.
\end{proof}
\begin{remark}  Two ucp asymptotic homomorphisms $\{\varphi^{(i)}_n:A \to M_{k^{(i)}_n}\}_n$, $i=0,1$, are called stably unitarily equivalent if there are two sequences of unital $*$-homomorphisms $\{\pi^{(i)}_n:A \to M_{r^{(i)}_n}\}_n$, $i=0,1$, and a sequence of  unitaries $(u_n)$ such that $$\lim_n \|\varphi^{(0)}_n (a)\oplus \pi^{(0)}_n(a) -u_n(\varphi^{(1)}_n (a)\oplus \pi^{(1)}_n(a))u_n^*\|=0$$ for all $a \in A$.
Let $A$ be a separable exact residually finite dimensional $C^*$-algebra satisfying the UCT. Arguing as in the proof of Theorem~\ref{thm:amen-alg} one shows that
    two ucp asymptotic homomorphism $\{\varphi^{(i)}_n:A \to M_{k^{(i)}_n}\}_n$, $i=0,1$, are  stably unitarily equivalent
 if and only if $(\varphi^{(0)}_n)_\sharp(x)\equiv (\varphi^{(1)}_n)_\sharp(x)$ for all $x \in K_0(A;\Z/p)$, $p \in \mathcal{P}$.
\end{remark}

\emph{Proof of Theorem~\ref{thm:3}:}

(i)  Higson and Kasparov \cite{HigKas:BC} proved that the Baum-Connes map is an isomorphism and that the canonical map $q_G:C^*(G)\to C^*_r(G)$ is a KK-equivalence. Moreover, Tu \cite{Tu:gamma} showed that $C^*(G)$ satisfies the UCT.
  Suppose now that $G$ is weakly matricially stable. Then $G$ is MAP by Lemma~\ref{lemma:qdmap} and hence quasidiagonal. By Proposition~\ref{cor:qd-embed} the canonical map $q_G:C^*(G)\to C^*_r(G)$ factors as:
$C^*(G) \stackrel{q_D}\longrightarrow D  \to C_r^*(G)$ with $D$ unital and quasidiagonal.
Since $q_G$ is a KK-equivalence, $q_D^*:K^0(D) \to K^0(C^*(G))$ is surjective. Using Remark~\ref{remark:qdd} we see that
\[ K^0(C^*(G))=q_G^* (K^0(C_r^*(G)))=q_D^*(K^0(D))= q_D^*(K^0(D)_{qd})\subseteq K^0(C^*(G))_{qd}.\]
Thus $C^*(G)$ is K-quasidiagonal. We conclude the argument by applying Theorem~\ref{thm-gen}.

(ii) Suppose now that $G$ is torsion free. By Theorem~\ref{thm:2}, $H^{even}(G;\Q)\cong \Q$. Since the Baum-Connes map is an isomorphism, we have   $K_0(C^*(G))\otimes \Q\cong RK_0(BG)\otimes\Q \cong H_{even}(G,\Q)\cong \Q$.

(iii) This was discussed in Remark~\ref{rem:flat}.\qed
\vskip 4pt
\emph{Proof of Theorem~\ref{thm:amen}:}
One direction follows Theorem~\ref{thm:3} since amenable groups have Haagerup's property.
The converse follows by applying Theorem~\ref{thm:amen-alg} to the nuclear $C^*$-algebra $C^*(G)$. Since $G$ is amenable, $C^*(G)$ is residually finite dimensional if and only if $G$ is MAP, see Proposition~\ref{Bk}.\qed

\section{Almost flat K-theory}\label{5}

Connes, Gromov and Moscovici \cite{CGM:flat} used the approximate monodromy correspondence which associates a group quasi-representation to an almost flat bundle
to prove the Novikov conjecture for large classes of groups.  Gromov indicates in \cite{Gromov:reflections, Gromov:curvature} how one constructs
  nontrivial almost flat $K$-theory classes for residually finite groups
 that are fundamental groups of even {dimensional}
  non-positively curved compact manifolds. We introduced in \cite{BB}, \cite{AA} an approach to almost flat K-theory via quasidiagonality of K-homology classes. This topic was further explored  in \cite{DadCarrion:almost_flat} where we discussed continuity properties of the approximate modronomy correspondence. The results of \cite{AA} and \cite{DadCarrion:almost_flat} were extended in \cite{Kubota1, Kubota2} to a relative setting for pair of groups which are residually amenable.

Let $Y$ be a compact Hausdorff space and let $(U_i)_{i\in I}$ be a fixed finite open cover of $Y$.
 A complex vector bundle on $Y$ of rank $m$ is
called $\varepsilon$-flat if is represented by a cocycle $v_{ij}:U_i\cap U_j \to U(m)$ such that
$\|v_{ij}(y)-v_{ij}(y')\|<\varepsilon$ for
all $y,y'\in U_i\cap U_j$ and all $i,j \in I$.
An element $x\in K^0(Y)$ is called almost flat if for any $\varepsilon>0$
there are $\varepsilon$-flat vector bundles $E,F$ such that $x=[E]-[F]$.
This property  does not depend on the cover $(U_i)_{i\in I}$.

  An element $x\in RK^0(BG)$ is called locally almost flat, if for any finite CW-subcomplex $j:Y \hookrightarrow BG$, the element
  $j^*(x)\in K^0(Y)$ is  almost flat.   We say that $x\in RK^0(BG)$ has  almost flat local multiples if for any $Y$ as above there is $m\geq 1$ such that  $m\cdot j^*(x)\in K^0(Y)$ is  almost flat. For a discussion of almost flatness, see (for example) \cite{AA}, \cite{DadCarrion:almost_flat}.

   By combining the methods of \cite{AA} and \cite{Kubota1},\cite{Kubota2}
   one obtains the following
  general result on the existence of almost flat K-theory classes. A rational  version of this result for residually amenable groups
  (including the relative case of pairs of groups) is due to Kubota \cite[Cor.5.14]{Kubota2}.
\begin{theorem}\label{thm:af}
  Let  $G$ be a countable  discrete quasidiagonal group that  admits a $\gamma$-element.
  Then all the elements of $RK^0(BG)$  have almost flat local multiples. If in addition $G$ is torsion free, then all the elements of $RK^0(BG)$ are locally almost flat.
\end{theorem}
\begin{proof} Suppose first that $G$ is torsion free. Then  \(\nu(KK(C^*(G),\C)_{qd})= RK^0(BG)\)  by Theorem~\ref{cor:qd-embed-kk}.
   By \cite[Cor.4.4]{AA}, for any finite CW-subcomplex $j:Y \hookrightarrow BG$, $j^*(\nu(KK(C^*(G),\C)_{qd})$ consists of entirely of almost flat elements. This last property is explained by the fact that the classes $\nu_Y(x)$ in \eqref{eqn:important} are almost flat
  as discussed in \cite{AA} and with more details in \cite{DadCarrion:almost_flat}. The degree of flatness of the right hand side of \eqref{eqn:important} is given by the degree to which the maps $\varphi^{(i)}_n$ are multiplicative.
  The first part of the statement is proved similarly using the fact that $\nu(KK(C^*(G),\QQ)_{qd})= RK^0(BG;\Q)$ as seen in the proof of Theorem~\ref{thm:2} and the observation that if $Y$ is a finite CW complex, then $RK^0(Y;\Q)\cong RK^0(Y)\otimes \Q$.
\end{proof}

 \textbf{Acknowledgements}
 I would like to thank Iason Moutzouris and Forrest Glebe for discussions that led to Example~\ref{qd-not-ra}.
\bibliographystyle{abbrv}

\end{document}